\documentclass[12pt,reqno]{amsart}
\usepackage{amssymb}
\usepackage{mathrsfs}
\usepackage{palatino}
\usepackage{bbm}
\usepackage{mathrsfs}
\usepackage{cases}
\usepackage{epic}
\usepackage{amsfonts}
\usepackage{graphicx}
\usepackage{amsmath}
\usepackage{amssymb, upgreek}
\usepackage{bm}
\usepackage{latexsym,todonotes}
\usepackage{pdflscape}
\usepackage[all]{xypic}
\usepackage[all]{xy}
\usepackage{color}
\usepackage{colordvi}
\usepackage{multicol}
\usepackage[normalem]{ulem}
\usepackage[linktocpage=true]{hyperref}
\usepackage{hyperref}
\hypersetup{colorlinks,linkcolor=blue,urlcolor=cyan,citecolor=red}
\numberwithin{equation}{section}

\newcommand{\CC}{\mathbb{C}}

\newcommand{\ZZ}{\mathbb{Z}}
\newcommand{\fg}{\mathfrak{g}}
\newcommand{\gl}{\mathfrak{gl}}
\newcommand{\kk}{\mathbbm{k}}
\newcommand{\T}{\mathrm{T}}

\newcommand{\Ymn}{Y_{M|N}}

\DeclareMathOperator{\Char}{char}

\DeclareMathOperator{\diag}{diag}

\DeclareMathOperator{\ev}{ev}

\DeclareMathOperator{\gr}{gr}

\newtheorem{theorem}{Theorem}[section]

\newtheorem{lemma}[theorem]{Lemma}
\newtheorem{proposition}[theorem]{Proposition}
\newtheorem{corollary}[theorem]{Corollary}
\theoremstyle{definition}
\newtheorem{definition}[theorem]{Definition}

\newtheorem{remark}[theorem]{ Remark}
\numberwithin{equation}{section}

\newcommand{\del}{\delta}
\newcommand{\Del}{\Delta}
\newcommand{\W}{\mathcal{W}}
\newcommand{\zomn}{$0^M1^N$}
\newcommand{\so}{\mathfrak{s}}
\newcommand{\mfs}{\check{\mathfrak{s}}}
\newcommand{\sd}{\mathfrak{s}^\dagger}
\newcommand{\sr}{\mathfrak{s}^r}
\newcommand{\glmn}{{\mathfrak g\mathfrak l}_{M|N}}

\newcommand{\ovl}[1]{\overline{#1}}

\numberwithin{equation}{section}

\begin{document}
\title[Parabolic presentations]{Parabolic presentations of the modular Super Yangian $\Ymn$ for arbitrary \zomn-sequences}
\author[Hongmei Hu]{Hongmei Hu}
\address[Hongmei Hu]{Department of Mathematics, Shanghai Maritime University, Shanghai 201306, China.}
\email{hmhu@shmtu.edu.cn}
\date{\today}

\makeatletter
\makeatletter
\@namedef{subjclassname@2020}{%
  \textup{2020} Mathematics Subject Classification}
\makeatother
\makeatother
\begin{abstract}
Let $\mu$ be an arbitrary composition of $M+N$ and let $\so$ be an arbitrary \zomn-sequence.
The present paper is devoted to extend parabolic presentations, depending on $\mu$ and $\so$, 
of the super Yangian $\Ymn$  to a field of positive characteristic.
\end{abstract}
\maketitle
\setcounter{tocdepth}{2}
\tableofcontents

\bigskip
\section{Introduction}
The Yangians, defined by Drinfeld \cite{Dr1, Dr2}, are certain non-commutative Hopf algebras which provides fundamental and important examples of quantum groups.
Also,
taking adventage of rational solutions of the {\em Yang-Baxter equation}, 
there is another way to present Yangians by the RTT formalism  in \cite[ch.11]{ES98} or \cite{FRT90},
which opens the way to their applications in statistical mechanics and mathematical physics.
Nowadays, the study of Yangians gives many new points of view and important applications to classical Lie theory; see the book \cite{Mo} and references therein.

In characteristic zero,
the Yangian $Y_N=Y(\mathfrak{gl}_N)$ of the Lie
algebra $\mathfrak{gl}_N$ was earlier considered in the works of mathematical physicists from
St.-Petersburg; see for instance \cite{TF79}.
The super Yangian $\Ymn$ associated to the general linear Lie superalgebra $\glmn$ is defined by Nazarov \cite{Na} in terms of the RTT presentation.
The {\it Drinfeld-type presentation} of $\Ymn$ was found by Gow in \cite{Gow07}.
Brundan and Kleshchev \cite{BK1} established a new presentation,
called as {\it parabolic presentations}, for the Yangian $Y_n$ associated to each composition $\mu$ of $n$. 
Then Peng Y. \cite{Peng16} built the generalization of \cite{BK1} for arbitrary $\mu$ and $01$-sequence to the super analogue of $Y_{n}$ over $\CC$,
the super Yangian $\Ymn$ associated to the general linear Lie superalgebra $\glmn$.

The main goal of this article is to extend the ones in \cite{Peng16} to an algebraically closed field $\kk$ of positive characteristic.
It aims to obtain {\it parabolic presentations} of the modular super Yangian $\Ymn$ for an {\em arbitrary} fixed 01-sequence $\so$ of $\glmn$ and an {\em arbitrary} fixed composition $\mu$ of $M+N$.
Then the fact that the super phenomenon could happen everywhere since $\mu$ and $\so$ are both arbitrary,
still exists in our setting.
The generators are commute over $\mathbb{C}$ if they are from two different blocks and the blocks are not ``close",
then enough relations for the general case are still obtained, in terms of the homomorphisms $\psi_{L}$ and $\zeta_{M|N}$, from the ones in  the less complicated situations that the length of $\mu$ is $2,3,4$.
Fortunately,
it remains to be true in the positive characteristic field $\kk$.
The author hope that the present paper  provides a self-contained reference when one is interested in parabolic presentations of super Yangians over a positive characteristic field, 
then all detailed proofs are offered here,
which, therefore,
confirm the statement that all most relations in the field $\kk$ are same with the ones in \cite{Peng16}.

Differenting from the proofs in \cite{Peng16},
we prove all relations related to parabolic generators $F$'s,
and find that the signs related to $F$'s arsing from the $\mathbb{Z}_2$-grading are more complicated than the ones related to $E$'s in \cite{Peng16},
so one needs more extra care here when treating the sign issues.
Also, we are in characteristic $p:=\Char{\kk}>0$,
and then the behaviour of these algebras will be different to the ordinary setting,
such as,
the argument that the identity $a=-a$ produces $a=0$ used in \cite{Peng16} will be false here.  
Then some identities in section \ref{sec:general} are the same as the ones in \cite{Peng16},
but we need to prove them again in different ways.
Extra care also are much more involved when treating the issues arising from sign factors in the field $\kk$.

Compared to these details proofs,
if ones just concern the parabolic realization of the modular super Yangian $Y_{M|N}$ for arbitrary $\mu$ and $\so$,
and their differences to the ones in the complex field setting,
then Theorem \ref{parabolic} is enough.
Actually, 
we have more relations in the positive characteristic field $\kk$ than \cite[Theorem 7.2]{Peng16}.
The main points are as follows.
\begin{enumerate}
\item If $p\neq 2$, the quartic Serre relations
for $|g|_3+|f|_2=0$ in parabolic presentations already follow from the quadratic and cubic relations,
however we cannot do that in the case $p=2$. 
Then we need here all  $F_{a+1;g,f}^{(1)}$ in the {\it quartic Serre relations} \eqref{coeffi-superserre-E} and \eqref{coeffi-superserre-F},
which is different from \cite[Theorem 7.2]{Peng16}.
Note that it is enough to consider the case $f_2=h, g_1=j$ in the relations (7.15)-(7.16) in \cite[Theorem 7.2]{Peng16}.
\item The relations (7.13)-(7.14) of {\it loc. cit.} are expressed here as the four relations \eqref{coeffi-serre-E}-\eqref{coeffi-super-F}.
	This is essential in characteristic $2$.
\item We add the case $a=c$ here in the relation \eqref{inj-4}, 
and different ways to prove the relations  \eqref{inj-1}--\eqref{inj-4}(\,cf.\,\cite[Lemma 7.4,(a)--(d)]{Peng16}) are provided here.
\end{enumerate}

A concrete realization of finite $W$-(super)algebras associated to {\em any} nilpotent element  in terms of (super)Yangians in \cite{BK2,Peng21}, 
demonstrate completely the importance of parabolic presentations.
The connection between Yangians and finite $W$-algebras was observed earlier in \cite{RS} for some particular nilpotent elements (called {\em rectangular} elements) with a different approach. Moreover, such a realization of finite $W$-(super)algebras provides a way to study their representation theory  by the representation theory of Yangians; see \cite{BK3}.
But Premet's work on the Kac-Weisfeiler
conjecture \cite{Premet} tell us that his motivation for finite $W$-algebras came from the representation
theory of Lie algebras in positive characteristic.
Recently, the precise analog of finite W-algebras in positive characteristic was introduced in \cite{GT19},
and their relations with Yangian were studied in \cite{GT21}.
Based on these, 
we mention one undergoing application of our result, 
which can also be thought as the true motivation.
One may try to build up the super version of \cite{GT21},
and generalize the argument in \cite{Peng21} to a positive characteristic field,
 so that a realization of (restricted) modular finite  $W$-superalgebras of type A in terms of the super Yangian $Y_{M|N}$ can be obtained. 

This article is organized as follows. We recall some basic properties of modular super Yangian $Y_{M|N}$ in Section 2. 
In Section 3, we take advantage of Gauss decomposition to define the parabolic generators. 
In Section 4, we  generalize some important homomorphisms between super Yangians to the positive characteristic field $\kk$ and prove them in detail.
It makes us reduce the general case to some less complicated special cases studied in Section 5 and 6. 
Section 7 is devoted to formulate and prove the main theorem.
\bigskip

\emph{Notation.}Throughout the article, 
$\kk$ denotes an algebraically closed field of characteristic $\Char(\kk)=:p>0$,
and we work over the ground field $\kk$.
We write $\glmn$ consisting of all $(M+N)\times (M+N)$ matrices over $\kk$ as the general linear Lie superalgebra.
$[a,b]$ means the set $\{a,a+1,\cdots,b\}, a,b\in\ZZ_{\geq 0}$,
and $e_{i,j}$ the elementary $ij$-matrix.
\section{Generalities on the Modular super Yangian $\Ymn$}
\subsection{$01$-sequence}
As is well known,
the supertrace on the Lie superalgebra $\glmn$  with the parity $|i|$ of $i$ is $0$ if $i\in[1,M]$, and $|i|=1$ if $i\in[M+1,M+N]$, 
gives rise to a non-degenerate supersymmetric bilinear form.
Restricting to the Cartan subalgebra $\mathfrak{h}$ of diagonal matrices,
we obtain a non-degenerate symmetric bilinear form on $\mathfrak{h}$.
Then the dual basis of $\{e_{k,k}, k\in[1,M+N]\}$ is
$\{\delta_{i}, \epsilon_{j}|i\in[1,M], j\in[M+1, M+N]\}$ when we identify $\delta_{i}$ with $(e_{ii},~)$, 
and $\epsilon_{j}$ with $-(e_{jj},~)$,
respectively. 
The  {\it standard positive system} of the root system corresponding to the {\it standard Borel subalgebra} of $\glmn$ consisting of upper triangular matrices, 
is given by $\{\epsilon_{i}-\epsilon_{j}|j\in [1,M+N], i<j\}$ with keeping in mind $\epsilon_{i}=\delta_{i}$ for $i\in[1,M]$,
and then its {\it standard fundamental system} is
\begin{equation}\label{fund}
\Pi^{st}=	\{\delta_{i}-\delta_{i+1},\delta_{M}-\epsilon_{1}, \epsilon_{j}-\epsilon_{j+1}|~i\in[1,M-1],~ j\in[M+1,M+N-1]\}.
\end{equation}
With this standard choice, $\delta_{M}-\epsilon_{1}$ is only one odd simple root.

It follows from the well-known classification for $\mathfrak{gl}_{M+N}$ 
that a fundamental system 
\begin{equation}\label{seq}
	\{\epsilon_{i_1}-\epsilon_{i_2},\epsilon_{i_2}-\epsilon_{i_3},\cdots,
\epsilon_{i_{M+N-1}}-\epsilon_{i_{M+N}}\},
\end{equation}
where $\{i_1, i_2, \cdots,i_{M+N} \}=[1,M+N]$ of $\glmn$ is obtained by restoring the parity of the simple roots.
 There is an important fact \cite[Section 1.3]{CW}) that all $\epsilon\delta$-sequence \eqref{seq} (also called as $0^m1^n$ or $01$-sequence),
which is a sequence consisting of $m$ $\delta$'s and $n$ $\epsilon$'s,
to  parameterize the set of the conjugacy classes of simple systems of $\glmn$ under the Weyl group $S_M\times S_N$ action. 
For the standard choice \eqref{fund},
its corresponding $01$-sequence is
\begin{equation}\label{standard}\so^{st}=\stackrel{M}{\overbrace{0\ldots0}}\,\stackrel{N}{\overbrace{1\ldots1}}.\end{equation}

We refer the reader to \cite[Chapter 1.]{CW} for more details and further applications of 01-sequences.

\subsection{Restricted Lie superalgebra $\glmn[x]$}
Let $\glmn[x]$ be the {\it current superalgebra} $\glmn\otimes\kk[x]$ with the standard basis 
$\{e_{i,j}x^r:=e_{i,j}\otimes x^r|r\in\ZZ_{\geq 0}, i,j\in[1,M+N]\}$,
and $\deg e_{i,j}x^r=|i|+|j|$. 
The supercommutation relations of $\glmn[x]$ are:
\begin{align*}\label{Lie bracket of g}
	[e_{i,j}x^r,e_{k,l}x^s]=\delta_{k,j}e_{i,l}x^{r+s}-(-1)^{(|i|+|j|)(|k|+|l)|}\delta_{l,i}e_{k,j}x^{r+s},\, \forall~ r,s\geq 0.
\end{align*}
We will always denote the Lie algebra by $\fg$ and write $U(\fg)$ for its enveloping algebra.
$U(\fg)$ has the natural filtration and grading with ${\rm deg}e_{i,j}x^r=r$.

The current superalgebra $\fg$ is a restricted Lie superalgebra with $p$-map defined on the basis by the rule $(ax^r)^{[p]}:=a^{[p]}x^{rp}$,
where $a^{[p]}$ denotes the $p\mathrm{th}$ matrix power of $a$,
here, 
$a$ belongs to the even part $\mathfrak{gl}_M\oplus\mathfrak{gl}_N$ of $\glmn$.

More details and results for the restricted Lie superalgebra $\glmn[x]$
can be found in \cite{CH}.

\subsection{Modular super Yangian $\Ymn$}\label{modular superYangians}
Corresponding to the case when $(01)$-sequence is the standard one $\so^{st}$  \eqref{standard},
then the original one in \cite{Na} gave rise to the definition \cite[Section 3 ]{CH} of the modular super Yangian $\Ymn$ over $\kk$ associated to the general linear Lie superalgebra $\mathfrak{gl}_{M|N}$.
There also is a {\it PBW theorem} for the modular super Yangian $\Ymn$ related to	 $\so^{st}$.
Generally,
we have the following 
\begin{definition}
For any fixed \zomn-sequence $\so$,
the modular super Yangian associated to the general linear Lie superalgebra $\mathfrak{gl}_{M|N}$,
denoted by $\Ymn$ hereafter,
is the associated superalgebra over $\kk$ with the {\it RTT generators} $\{t_{i,j}^{(r)};~1{\leq}i,j{\leq}M+N, r\geq 1\}$ subject to the following relations:
\begin{equation}\label{RTT relations}
	\left[t_{i,j}^{(r)}, t_{k,l}^{(s)}\right] =(-1)^{|i||j|+|i||k|+|j||k|}\sum_{t=0}^{\min(r,s)-1}
	\left(t_{k, j}^{(t)} t_{i,l}^{(r+s-1-t)}-
	t_{k,j}^{(r+s-1-t)}t_{i,l}^{(t)}\right),
\end{equation}
where $|i|$ denote the $i$-th digit of $\so$,
and the parity of $t_{i,j}^{(r)}$ for $r>0$ is defined by $|i|+|j|~(\text{mod}~2)$,
and the bracket in \eqref{RTT relations} is the supercommutator.
The element $t_{i,j}^{(r)}$ is called an {\it even} ({\it odd}, respectively) element if its parity is $0$ ($1$, respectively).
By convention, we set $t_{i,j}^{(0)}:=\delta_{i,j}$.
\end{definition}
It is easy to obtain that the definition of $\Ymn$ is is independent of the choices of $\so$, up to an isomorphism,
so we can omit it in the notation.

Define the $(M+N)\times (M+N)$ matrix $T(u)=:\Big( t_{i,j}(u) \Big), 1{\leq}i,j{\leq}M+N$ with entries in $\Ymn[[u^{-1}]]$,
where $t_{i,j}(u)$ is the formal power series 
\[
t_{i,j}(u):= \sum_{r\geq 0} t_{i,j}^{(r)}u^{-r} \in Y_{M|N}[[u^{-1}]],\,\forall \,i,j\in[1,M+N].
\]
Then we may rewrite equation (\ref{RTT relations}) into the following series form
\begin{align}\label{tiju tkl relation}
(u-v)[t_{i,j}(u),t_{k,l}(v)]=(-1)^{|i||j|+|i||k|+|j||k|}(t_{k,j}(u)t_{i,l}(v)-t_{k,j}(v)t_{i,l}(u)).
\end{align}

The matrix $T(u)$ is always invertible, 
so we define the entries of its inverse as
\begin{equation}\label{Tp=FDE}
	\big(T(u)\big)^{-1}:=\big(t_{ij}^{\prime}(u)\big)_{i,j=1}^{M+N}.
\end{equation}

Then a straightforward calculation yield the following relation, which will be used frequently later (see \cite[(5)]{Gow07}, \cite[(3.12)]{Peng16}).
\begin{align}\label{commurelation}
	(u-v)[t_{i,j}(u),t'_{k,l}(v)]{=}(-1)^{|i||j|+|i||k|+|j||k|}(\delta_{k,j}\sum\limits_{s=1}^{m{+}n}t_{i,s}(u)t'_{s,l}(v){-}\delta_{i,l}\sum\limits_{s=1}^{m{+}n}t'_{k,s}(v)t_{s,j}(u)).
\end{align}

The modular super Yangian $\Ymn$ is a Hopf-superalgebra observed in \cite[p.125]{Na},
with the comultiplication 
$\Del:\Ymn\rightarrow \Ymn\otimes \Ymn$ is given by 
\begin{equation}\label{Del}
	\Del(t_{i,j}^{(r)})=\sum_{s=0}^r \sum_{k=1}^{M+N} t_{i,k}^{(r-s)}\otimes t_{k,j}^{(s)},\, \forall i,j\in[1,M+N].
\end{equation}
Moreover, there exists a surjective homomorphism $$\ev:\Ymn\rightarrow U(\glmn)$$ called the {\em evaluation homomorphism}, defined by
\begin{equation}\label{ev}
	\ev\big(t_{i,j}(u)\big):= \del_{ij} + (-1)^{|i|} e_{ij}u^{-1},
\end{equation}

Also,
there is {\it PBW theorem} \cite[section 3]{CH} for the modular super Yangian $\Ymn$ works perfectly for a fixed $01$-sequence $\so$.
\begin{proposition}
Ordered supermonomial in the generators $\{t_{i,j}^{(r)};~1{\leq}i,j{\leq}M+N, r\geq 1\}$ taken in some fixed order forms a linear basis for $\Ymn$.
\end{proposition}

Moreover,
the {\it loop filtration}
\begin{equation}\label{loop filtration}
	\Ymn=\bigcup_{r\geq 0}{\rm F}_r\Ymn
\end{equation}
by setting ${\rm deg}t_{i,j}^{(r)}=r-1$,
i.e., ${\rm F}_r\Ymn$ is the span of all supermonomials of the form $t_{i_1,j_1}^{(r_1)}\dots t_{i_m,j_m}^{(r_m)}$ with $(r_1-1)+\cdots+(r_m-1){\leq}r$,
to make sure that the assignment
\begin{equation}\label{grading-iso}{\rm gr}_{r-1}t_{i,j}^{(r)}\mapsto (-1)^{|i|}e_{i,j}x^{r-1}\end{equation}
gives  an isomorphism ${\rm gr} \Ymn\cong U(\fg)$ of graded superalgebras.

\section{Parabolic generators in positive characteristic}\label{section: parabolic}
Basically,
the techniques in \cite[section3]{Peng16} work perfectly in any positive characteristic,
but to make this article self-contained,
we will spend some time explaining notation precisely.

Let $\mu=(\mu_1,\ldots,\mu_n)$ be a given composition of $M+N$ with length $n$,
$\so$ a fixed \zomn-sequence. 
We break $\so=\so_1\so_2\cdots\so_n$ into $n$ subsequences according to $\mu$,
where $\so_1$ is the subsequence consisting of the first $\mu_1$ digits of $\so$, $\so_2$ is the subsequence consisting of the next $\mu_2$ digits of $\so$, and so on.

For each $a\in[1,n]$, let $p_a$ and $q_a$ denote the number of $0$'s and $1$'s in $\so_a$, respectively.
Then each $\so_a$ is a $0^{p_a}1^{q_a}$-sequence of $\gl_{p_a|q_a}$,
and for each $i{\leq}[1,\mu_a]$, define the {\em restricted parity} $|i|_a$ by
\begin{equation}\label{respa} 
|i|_a:=\mbox{ the~} i-\mbox{th~ digits~ of~} \so_a=|\sum\limits_{j=1}^{a-1}\mu_j+i|. 
\end{equation}

Suppose that $A, B, C$ and $D$ are $a \times a$, $a \times b$, $b \times a$ and $b \times b$ matrices respectively with entries in some ring.
Assuming that the matrix $A$ is invertible, we have the following definition
by {\it quasideterminants} in \cite{GR} and the notation in \cite[(4.3)]{BK1}.
\begin{equation*}
	\left|
	\begin{array}{cc}
		A&B\\
		C&
		\hbox{\begin{tabular}{|c|}\hline$D$\\\hline\end{tabular}}
	\end{array}
	\right| := D - C A^{-1} B.
\end{equation*}

For any fixed $\so$, the leading minors of the matrix $T(u)$ are invertible, It follows that it possesses the $Gauss$ $decomposition$ \cite{GR} associated to $\mu$:
\begin{equation}\label{gaussdec}
	T(u) = F(u) D(u) E(u)
\end{equation}
for unique {\em block matrices} $D(u)=\diag(D_1(u),\cdots,D_n(u))$, and $E(u),F(u)$ of the form
\begin{gather*}
E(u) =
\left(
\begin{array}{cccc}
	I_{\mu_1} & E_{1,2}(u) &\cdots&E_{1,n}(u)\\
	0 & I_{\mu_2} &\cdots&E_{2,n}(u)\\
	\vdots&\vdots&\ddots&\vdots\\
	0&0 &\cdots&I_{\mu_{n}}
\end{array}
\right),\quad
F(u) = \left(
\begin{array}{cccc}
	I_{\mu_1} & 0 &\cdots&0\\
	F_{2,1}(u) & I_{\mu_2} &\cdots&0\\
	\vdots&\vdots&\ddots&\vdots\\
	F_{n,1}(u)&F_{n,2}(u) &\cdots&I_{\mu_{n}}
\end{array}
\right),
\end{gather*}
where
\begin{gather*}
	D_a(u) =\big(D_{a;i,j}(u)\big)_{1 {\leq}i,j {\leq}\mu_a},\\
	E_{a,b}(u)=\big(E_{a,b;i,j}(u)\big)_{1 {\leq}i {\leq}\mu_a, 1 {\leq}j {\leq}\mu_b},\quad
	F_{b,a}(u)=\big(F_{b,a;i,j}(u)\big)_{1 {\leq}i {\leq}\mu_b, 1 {\leq}j {\leq}\mu_a},
\end{gather*}
are $\mu_a \times \mu_a$,
$\mu_a \times \mu_b$
and  $\mu_b \times\mu_a$ matrices, respectively.
Also note that all the submatrices $D_{a}(u)$'s are invertible, and we define the $\mu_a\times\mu_a$ matrix
$D_a^{\prime}(u)=\big(D_{a;i,j}^{\prime}(u)\big)_{1{\leq}i,j{\leq}\mu_a}$ by
$
	D_a^{\prime}(u):=\big(D_a(u)\big)^{-1}.
$
The entries of these matrices are certain power series
\begin{eqnarray*}
&D_{a;i,j}(u) =& \sum_{r \geq 0} D_{a;i,j}^{(r)} u^{-r},\quad
	D_{a;i,j}^{\prime}(u) = \sum_{r \geq 0}D^{\prime(r)}_{a;i,j} u^{-r},\\
	&E_{a,b;i,j}(u) =& \sum_{r \geq 1} E_{a,b;i,j}^{(r)} u^{-r},\quad
	F_{b,a;i,j}(u) =\sum_{r \geq 1} F_{b,a;i,j}^{(r)} u^{-r}.
\end{eqnarray*}
In addition, for $1{\leq}a{\leq}n-1$, we set
\begin{eqnarray*}
	E_{a;i,j}(u) := E_{a,a+1;i,j}(u)=\sum_{r \geq 1}  E_{a;i,j}^{(r)} u^{-r},\quad
	F_{a;i,j}(u) := F_{a+1,a;i,j}(u)=\sum_{r \geq 1} F_{a;i,j}^{(r)} u^{-r}.
\end{eqnarray*}

Moreover,
when we write the matrix $T(u)$ in the block form according to $\mu$ as 
$$T(u) = \left(
\begin{array}{lll}
	{^\mu}T_{1,1}(u)&\cdots&{^\mu}T_{1,n}(u)\\
	\vdots&\ddots&\cdots\\
	{^\mu}T_{n,1}(u)&\cdots&{^\mu}T_{n,n}(u)\\
\end{array}
\right),
$$ 
where each ${^\mu}T_{a,b}(u)$ is a $\mu_a \times \mu_b$ matrix.
Then we have 
\begin{align}\label{quasid}
	&D_a(u) =
	\left|
	\begin{array}{cccc}
		{^\mu}T_{1,1}(u) & \cdots & {^\mu}T_{1,a-1}(u)&{^\mu}T_{1,a}(u)\\
		\vdots & \ddots &\vdots&\vdots\\
		{^\mu}T_{a-1,1}(u)&\cdots&{^\mu}T_{a-1,a-1}(u)&{^\mu}T_{a-1,a}(u)\\
		{^\mu}T_{a,1}(u) & \cdots & {^\mu}T_{a,a-1}(u)&
		\hbox{\begin{tabular}{|c|}\hline${^\mu}T_{a,a}(u)$\\\hline\end{tabular}}
	\end{array}
	\right|,\\[4mm]
	&E_{a,b}(u) =\label{quasie}
	D^{\prime}_a(u)
	\left|\begin{array}{cccc}
		{^\mu}T_{1,1}(u) & \cdots &{^\mu}T_{1,a-1}(u)& {^\mu}T_{1,b}(u)\\
		\vdots & \ddots &\vdots&\vdots\\
		{^\mu}T_{a-1,1}(u) & \cdots & {^\mu}T_{a-1,a-1}(u)&{^\mu}T_{a-1,b}(u)\\
		{^\mu}T_{a,1}(u) & \cdots & {^\mu}T_{a,a-1}(u)&
		\hbox{\begin{tabular}{|c|}\hline${^\mu}T_{a,b}(u)$\\\hline\end{tabular}}
	\end{array}
	\right|,\\[4mm]
	&F_{b,a}(u) =\label{quasif}
	\left|
	\begin{array}{cccc}
		{^\mu}T_{1,1}(u) & \cdots &{^\mu}T_{1,a-1}(u)& {^\mu}T_{1,a}(u)\\
		\vdots & \ddots &\vdots&\vdots\\
		{^\mu}T_{a-1,1}(u) & \cdots & {^\mu}T_{a-1,a-1}(u)&{^\mu}T_{a-1,a}(u)\\
		{^\mu}T_{b,1}(u) & \cdots & {^\mu}T_{b,a-1}(u)&
		\hbox{\begin{tabular}{|c|}\hline${^\mu}T_{b,a}(u)$\\\hline\end{tabular}}
	\end{array}
	\right|D^{\prime}_a(u),
\end{align}
for all $1{\leq}a{\leq}n$ in (\ref{quasid}) and $1{\leq}a<b{\leq}n$ in (\ref{quasie}), (\ref{quasif}).

Let $T_{a,b;i,j}(u)$ be the $(i,j)$-th entry of the $\mu_a\times\mu_b$ matrix $^\mu T_{a,b}(u)$ and let $T_{a,b;i,j}^{(r)}$ denote the coefficient of $u^{-r}$ in $T_{a,b;i,j}(u)$.
Acoording to the special case of when $a=1$ in \eqref{quasid},
and \eqref{quasie}, \eqref{quasif},
we have 
\begin{gather}
	D_{1;i,j}^{(r)}=T_{1,1;i,j}^{(r)}=t_{i,j}^{(r)},\quad \ 1{\leq}i,j{\leq}\mu_1, r\geq 0;\label{Dtident}\\
	E_{b-1;i,j}^{(1)}{=}T_{b-1,b;i,j}^{(1)}, \, F_{b-1;j,i}^{(1)}{=}T_{b,b-1;j,i}^{(1)},\quad 2{\leq}b{\leq}n, 1{\leq}i{\leq}\mu_{b-1}, 1{\leq}j{\leq}\mu_{b}.\label{Jacobi1}
\end{gather}

\begin{lemma}
	For each pair $a$, $b$ such that $1<a+1<b{\leq}n-1$ and $1{\leq}i\leq\mu_a$, $1 {\leq}j {\leq}\mu_b$, we have
	\begin{equation}
		E_{a,b;i,j}^{(r)}{=}(-1)^{|k|_{b-1}}[E_{a,b-1;i,k}^{(r)}, E_{b-1;k,j}^{(1)}],
		\quad
		F_{b,a;j,i}^{(r)}{=}(-1)^{|k|_{b-1}}[F_{b-1;j,k}^{(1)}, F_{b-1,a;k,i}^{(r)}],\label{ter}
	\end{equation} 
	for any fixed $1 {\leq}k {\leq}\mu_{b-1}$. 
\end{lemma}
\begin{remark}
The equivalent expressions of \eqref{ter} by the power series is:
\begin{align}
	E_{a,b;i,j}(u)&=(-1)^{|k|_{b-1}}[E_{a,b-1;i,k}(u), E_{b-1;k,j}^{(1)}],\label{terpow1}\\
	F_{b,a;j,i}(u)&=(-1)^{|k|_{b-1}}[F_{b-1;j,k}^{(1)}, F_{b-1,a;k,i}(u)].\label{terpow2}
\end{align}
\end{remark}
\begin{proof}
The inductive relation of the case when $b=3$ and $a=1$ for the $E$'s has been proved in detail in \cite[Lemma 3.3]{Peng16},
then in order to demonstrate that the proof therein work perfectly for the positive characteristic,
we will verify the statement of the $F$'s for the general case. 

We start firstly with the initial step when $b=3$ and $a=1$,
and then obtain the general case,
which is similar, except that the expression of $F_{b-1,a;k,i}^{(r)}$ is more complicated.
By the case of $a=1$ in \eqref{quasif} and \eqref{Jacobi1},
we have the following first equality
$$
\begin{array}{rl}
&[F_{3,2;j,k}^{(1)}, F_{2,1;k,i}^{(r)}]=[T_{3,2;j,k}^{(1)}, \sum_{s=0}^{r}\sum_{t=1}^{\mu_1}T^{(s)}_{2,1;k,t}D'^{(r-s)}_{1;t,i}]\\
=&[t_{\mu_1+\mu_2+j,\mu_1+k}^{(1)}, \sum_{s=0}^{r}\sum_{t=1}^{\mu_1}t^{(s)}_{\mu_1+k,t}D'^{(r-s)}_{1;t,i}]\\
=&\sum_{s=0}^{r}\sum_{t=1}^{\mu_1}[t_{\mu_1+\mu_2+j,\mu_1+k}^{(1)}, t^{(s)}_{\mu_1+k,t}]D'^{(r-s)}_{1;t,i}\\
=&\sum_{s=0}^{r}\sum_{t=1}^{\mu_1}(-1)^{\mu_1+k}\left(t_{\mu_1+\mu_2+j,t}^{(s)}-\delta_{\mu_1+\mu_2+j,t}t^{(s)}_{\mu_1+k,\mu_1+k}\right)D'^{(r-s)}_{1;t,i}\\
=&(-1)^{|k|_2}\sum_{s=0}^{r}\sum_{t=1}^{\mu_1}t_{\mu_1+\mu_2+j,t}^{(s)}D'^{(r-s)}_{1;t,i}
=(-1)^{|k|_2}\sum_{s=0}^{r}\sum_{t=1}^{\mu_1}T_{3,1;j,t}^{(s)}D'^{(r-s)}_{1;t,i}\\
=&(-1)^{|k|_2}F_{3,1;j,i}^{(r)},
\end{array}
$$ 
where $1{\leq}j{\leq}\mu_3$, $1{\leq}k{\leq}\mu_2$ and $1{\leq}t{\leq}\mu_1$.
Note that we may express $D_{1;t,i}^{\prime(r-s)}$ in terms of $t_{\alpha,\beta}^{\prime(r)}$, and the subscriptions $1{\leq}\alpha,\beta{\leq}\mu_1$ will never overlap with the subscriptions of $t_{\mu_1+\mu_2+j,\mu_1+k}^{(1)}$, 
it follows that they commute with each other by \eqref{commurelation}, 
and then the third equality is obtained. 

Generally,
according to \eqref{quasif},
we have 
\begin{equation}\label{1}
F_{m,a}(u){=}{^\mu}T_{m,a}(u)D^{\prime}_a(u)
{-}\sum\limits_{x,y=1}^{a-1}{^\mu}T_{m,x}(u)\cdot {^\mu}S_{x,y}(u)\cdot {^\mu}T_{y,a}(u)\cdot D^{\prime}_a(u)
\end{equation}
for all $a<m{\leq}n$,
where
$$S(u){=}\left(
\begin{array}{ccc}
	{^\mu}S_{1,1}(u) & \cdots &{^\mu}S_{1,a-1}(u)\\
	\vdots & \ddots &\vdots\\
	{^\mu}S_{a-1,1}(u) & \cdots & {^\mu}S_{a-1,a-1}(u)
\end{array}
\right):{=}\left(
\begin{array}{ccc}
	{^\mu}T_{1,1}(u) & \cdots &{^\mu}T_{1,a-1}(u)\\
	\vdots & \ddots &\vdots\\
	{^\mu}T_{a-1,1}(u) & \cdots & {^\mu}T_{a-1,a-1}(u)
\end{array}
\right)^{-1}.$$

In order to obtain the general case,
we focus on the super-bracket of $F_{m;j,k}^{(1)}$ and the right-hand side of the equation \eqref{1}.
Note that we may express every coefficient of $ {^\mu}S_{x,y}(u)$,
${^\mu}T_{y,a}(u)$
and $D^{\prime}_a(u)$ in terms of $t^{(s)}_{\alpha,\beta}$, $t'^{(s)}_{\alpha,\beta}$,
and then obtain the fact that their subscriptions $1{\leq}\alpha,\beta{\leq}\sum_{t=1}^{a}\mu_t<\sum_{t=1}^{m-1}\mu_t<\sum_{t=1}^{m}\mu_t$, 
will never overlap with the subscriptions of 
$F_{m;j,k}^{(1)}=T^{(1)}_{m+1,m;j,k}=t^{(1)}_{\sum_{t=1}^{m}\mu_t+j,\sum_{t=1}^{m-1}\mu_t+k}$.
Thus they supercommute by $\eqref{RTT relations}$ and $\eqref{commurelation}$.
For each coefficient $T^{(s)}_{m,x;k,i}$ of the $(k,i)$-entry in  ${^\mu}T_{m,x}(u)$,
we have 
\begin{align*}
&[F_{m;j,k}^{(1)},T^{(s)}_{m,x;k,i}]
=[T_{m+1,m;j,k}^{(1)},T^{(s)}_{m,x;k,i}]
=\left[t^{(1)}_{\sum_{t=1}^{m}\mu_t+j,\sum_{t=1}^{m-1}\mu_t+k},t^{(s)}_{\sum_{t=1}^{m-1}\mu_t+k,\sum_{t=1}^{x-1}\mu_t+i}\right]	\\
&=_{\eqref{RTT relations}}(-1)^{|k|_{m}}t^{(s)}_{\sum_{t=1}^{m}\mu_t+j,\sum_{t=1}^{x-1}\mu_t+i}
=(-1)^{|k|_{m}}T_{m+1,x;j,i}^{(s)},\\
\end{align*}
which gives
\begin{equation}\label{2}
[F_{m;j,k}^{(1)},T_{m,x;k,i}(u)]=(-1)^{|k|_{m}}T_{m+1,x;j,i}(u).
\end{equation}
Similartly,
considering the coefficients of the entries  in  ${^\mu}T_{m,a}(u)D^{\prime}_a(u)$,
we have 
\begin{align*}
&\left[F_{m;j,k}^{(1)},(\left({^\mu}T_{m,a}(u)D^{\prime}_a(u)\right)_{k,i}){}^{(r)}\right]
=\left[T_{m+1,m;j,k}^{(1)},\sum_{b=1}^{\mu_a}\sum_{s=0}^{r}T_{m,a;k,b}^{(s)}  D'{}^{(r-s)}{}_{a;b,i}\right]\\
=&\sum_{b=1}^{\mu_a}\sum_{s=0}^{r}\left[t_{\sum_{t=1}^{m}\mu_t+j,\sum_{t=1}^{m-1}\mu_t+k}^{(1)},t_{\sum_{t=1}^{m-1}\mu_t+k,\sum_{t=1}^{a-1}\mu_t+b}^{(s)}  D'{}^{(r-s)}_{\sum_{t=1}^{a-1}\mu_t+b,\sum_{t=1}^{a-1}\mu_t+i}\right]	\\
&\mbox{since~the~ subscriptions~}
\sum_{t=1}^{a-1}\mu_t<\sum_{t=1}^{m}\mu_t, \\
&=_{\eqref{commurelation}}\sum_{b=1}^{\mu_a}\sum_{s=0}^{r}\left[t_{\sum_{t=1}^{m}\mu_t+j,\sum_{t=1}^{m-1}\mu_t+k}^{(1)},t_{\sum_{t=1}^{m-1}\mu_t+k,\sum_{t=1}^{a-1}\mu_t+b}^{(s)} \right] D'{}^{(r-s)}_{\sum_{t=1}^{a-1}\mu_t+b,\sum_{t=1}^{a-1}\mu_t+i}\\
&=_{\eqref{RTT relations}}\sum_{b=1}^{\mu_a}\sum_{s=0}^{r}(-1)^{|k|_{m}}t^{(s)}_{\sum_{t=1}^{m}\mu_t+j,\sum_{t=1}^{a-1}\mu_t+b}D'{}^{(r-s)}_{\sum_{t=1}^{a-1}\mu_t+b,\sum_{t=1}^{a-1}\mu_t+i}\\
&=\sum_{b=1}^{\mu_a}\sum_{s=0}^{r}(-1)^{|k|_{m}}T_{m+1,a;j,b}^{(s)}D'{}^{(r-s)}_{a,a;b,i}
=(-1)^{|k|_{m}}(\left({^\mu}T_{m+1,a}(u)D^{\prime}_a(u)\right)_{j,i}){}^{(r)},
\end{align*}
which gives
\begin{equation}\label{3}
[F_{m;j,k}^{(1)},\left({^\mu}T_{m,a}(u)D^{\prime}_a(u)\right)_{k,i}]=(-1)^{|k|_{m}}\left({^\mu}T_{m+1,a}(u)D^{\prime}_a(u)\right)_{j,i}.
\end{equation}

With \eqref{1}--\eqref{3} and supercommutate,
we have
\begin{align*}
&[F_{b-1;j,k}^{(1)}, F_{b-1,a;k,i}(u)]\\
=&_{\eqref{1}}\left[F_{b-1;j,k}^{(1)}, \left({^\mu}T_{b-1,a}(u)D^{\prime}_a(u)\right)_{j,i}
{-}\sum\limits_{x,y=1}^{a-1}\left({^\mu}T_{b-1,x}(u)\cdot {^\mu}S_{x,y}(u)\cdot {^\mu}T_{y,a}(u)\cdot D^{\prime}_a(u)\right)_{k,i}\right]\\
=&^{\eqref{2}}_{\eqref{3}}(-1)^{|k|_{b-1}}\left({^\mu}T_{b,a}(u)D^{\prime}_a(u)_{j,i}{-}\sum\limits_{x,y=1}^{a-1}\left({^\mu}T_{b,x}(u)\cdot {^\mu}S_{x,y}(u)\cdot {^\mu}T_{y,a}(u)\cdot D^{\prime}_a(u)\right)_{j,i}\right)\\
=&(-1)^{|k|_{b-1}}\left({^\mu}T_{b,a}(u)D^{\prime}_a(u){-}\sum\limits_{x,y=1}^{a-1}\left({^\mu}T_{b,x}(u)\cdot {^\mu}S_{x,y}(u)\cdot {^\mu}T_{y,a}(u)\cdot D^{\prime}_a(u)\right)\right)_{j,i}\\
=&_{\eqref{1}}(-1)^{|k|_{b-1}}F_{b,a;j,i}(u).
\end{align*}
\end{proof}

For any fixed shape $\mu$,
the matrix equation \eqref{gaussdec} means that each generator $t_{ij}^{(r)}$ can be expressed as a sum of supermonomials in $D_{a;i,j}^{(r)}$, $E_{a,b;i,j}^{(r)}$ and $F_{b,a;i,j}^{(r)}$, in a certain order that all the $F$'s appear before the $D$'s and all the $D$'s appear before the $E$'s. By (\ref{ter}), it is enough to use $D_{a;i,j}^{(r)}$, $E_{a;i,j}^{(r)}$ and $ F_{a;i,j}^{(r)}$ only rather than use all of the $E$'s and the $F$'s.
So we have the following
\begin{theorem}\label{gendef}
For any fixed shape $\mu$,
the modular superalgebra $\Ymn$ is generated by the following elements
	\begin{align*}
		&\big\lbrace D_{a;i,j}^{(r)}, D_{a;i,j}^{\prime(r)} \,|\, {1{\leq}a{\leq}n,\; 1{\leq}i,j{\leq}\mu_a,\; r\geq 0}\big\rbrace,\\
		&\big\lbrace E_{a;i,j}^{(r)} \,|\, {1{\leq}a< n,\; 1{\leq}i{\leq}\mu_a, 1{\leq}j\leq\mu_{a+1},\; r\geq 1}\big\rbrace,\\
		&\big\lbrace F_{a;i,j}^{(r)} \,|\, {1{\leq}a< n,\; 1{\leq}i\leq\mu_{a+1}, 1{\leq}j{\leq}\mu_a,\; r\geq 1}\big\rbrace.
	\end{align*}
\end{theorem}
These generators in the above Theorem \ref{gendef} are called as \textit{parabolic generators} of $\Ymn$.
Note that parabolic generators of $\Ymn$ depend on the shape $\mu$,
and their parities depend on the fixed sequence $\so$.
The paper is devoted to write down explicit defining relations of 
$\Ymn$ over $\kk$ for any fixed $\mu$ and any fixed $\so$.
In the followings,
we will use the notation $Y_{\mu}$ or $\Ymn(\so)$ or $Y_{\mu}(\so)$ to emphasize the choice of $\mu$ or $\so$ or both when necessary.
\section{Maps between Super Yangians}\label{section:maps}
The relations among parabolic generators of $\Ymn$ are mainly rely on the special cases when $n=2,3$, 
since there are certain nice homomorphisms between super Yangian,
which carry the relations in the special cases to the general case.
There are some notations related to any fixed \zomn-sequence $\so$ in \cite[section 4]{Peng16}:
\begin{enumerate}
\item $\mfs$:= the \zomn-sequence obtained by interchanging the 0's and 1's of $\so$.
\item $\sr$:= the reverse of $\so$.
\item $\sd$:= $(\mfs)^r$, the reverse of $\mfs$.
\end{enumerate}
Let $|i|_{\so}$ be the parity of the $i$-th digit in the sequence $\so$.
Then it is easy to obtain	
\begin{equation}\label{parity}
|i|_{\so}=|i|_{\mfs}+1(\textbf{mod}2),
\,
|i|_{\so}=|M{+}N{+}1{-}i|_{\sr},\,
|i|_{\so}=|M{+}N{+}1{-}i|_{\sd}+1(\textbf{mod}2).
\end{equation}
The notation $\so_1\so_2$ means the concatenation of two $01$-sequences $\so_1$ and $\so_2$.
We can extend some homomorphisms of $\Ymn$ over $\CC$ to the positive characteristic in the following proposition.
Note that 
\begin{proposition}
	\begin{enumerate}
		\item The map $\rho_{M|N}:\Ymn(\so)\rightarrow Y_{N|M}(\sd)$ defined by
		\[
		\rho_{M|N}\big(t_{ij}(u)\big)=t_{M+N+1-i,M+N+1-j}(-u)
		\]
		is an isomorphism.
		\item The map $\sigma_{M|N}:\Ymn(\so)\rightarrow \Ymn(\so)$ defined by
		\[\sigma_{M|N}\big(T(u)\big)=T(-u) \,~i.e.~\, \sigma_{M|N}\big(t_{ij}(u)\big)=t_{ij}(-u) \]
		is an involutive anti-automorphism of $\Ymn(\so)$.
		\item The map $S_{M|N}:\Ymn(\so)\rightarrow \Ymn(\so)$ defined by
		\[S_{M|N}\big(T(u)\big)=\big(T(u)\big)^{-1} \,~i.e.~\, S_{M|N}\big(t_{ij}(u)\big)=t^{\prime}_{ij}(u) \]
			is an anti-automorphism of $\Ymn(\so)$.
		\item  The map $\omega_{M|N}:\Ymn(\so)\rightarrow \Ymn(\so)$ defined by taking the composition of the anti-homomorphism $\sigma_{M|N}$ and $S_{M|N}$, i.e.,
		\[
		\omega_{M|N}\big(T(u)\big)=\big(T(-u)\big)^{-1}
		\,~i.e.~\, \omega_{M|N}\big(t_{ij}(u)\big)=t^{\prime}_{ij}(-u)
		\]
		is an involutive automorphism.
		\end{enumerate}
\end{proposition}
\begin{proof}
The map $\rho_{M|N}$ follows easily from the defining relations \eqref{tiju tkl relation} associated to the relation of parities in \eqref{parity}.
As is well known,
these relations \eqref{tiju tkl relation} can be equivalently written in the following $RTT$ matrix form:
\begin{equation}\label{RTT}
R(u-v)T_1(u)T_2(v)=T_2(v)T_1(u)R(u-v)
\end{equation}	
where $R(u)$ is the \textit{Yang $R$-matrix} of $\textbf{1}-Pu^{-1}$,
and $P=\sum_{i,j=1}^{M+N}(-1)^{|i|}e_{ij}\otimes e_{ji}$.

Then,
if we conjugate both sides of \eqref{RTT} by $P$ and then replace $(u,v)$ by $(-v,-u)$,
then we have
$$R(u-v)T_2(-v)T_1(-u)=T_1(-u)T_2(v)R(u-v).$$
So the map $\sigma_{M|N}$ is obtained.
Note now that the map $\sigma_{M|N}$ is involutive,
and then it is bijective.
For the anti-homomorphism of the map $S_{M|N}$ observe that the relation
$$R(u-v)T_2(v)^{-1}T_1(u)^{-1}=R(u-v)T_2(v)^{-1}T_1(u)^{-1}$$
is equivalent to \eqref{RTT}.
With these,
then the map $\omega_{M|N}$ is an algebra homomorphism.
Let us show that it is involutive.
Applying $\omega_{M|N}$ to both sides of the identity $\omega_{M|N}(T(u))\cdot T(-u)=1$,
it gives $\omega^2_{M|N}(T(u))\cdot T(u)^{-1}=1$,
so the square $\omega^2_{M|N}$ is the identity map.
It follows that the map $\omega_{M|N}$ is bijective,
and then the map $S_{M|N}$ also is bijective.
\end{proof}
Notice that the anti-automorphism $S_{M|N}$ is \textit{not} involutive.
It is the antipode of the Hopf algebra $\Ymn$.

The following {\it swap map} $\zeta_{M|N}$ in proposition \ref{swap} and {\it shift map} $\psi_{p|q}$ in proposition \ref{shift},
have an important role in presenting Yangians in \cite{BK1}, \cite{Gow07} and \cite{Peng16}. 
In order to state that this will also be displayed again in the positive characteristic,
we will describe them explicitly.

${}^{\mu}\widetilde E_{a,b}(u)$ and ${}^{\mu}\widetilde F_{a,b}(u)$ denote the $\mu_a\times\mu_b$-matrices in the $(a,b)$-th block position of the upper-triangular block matrix $E(u)^{-1}$ and the lower-triangular block matrix $F(u)^{-1}$, respectively.
Let $\widetilde E_{a,b;i,j}(u), \widetilde F_{a,b;i,j}(u) $ be the $(i,j)$-th entry of the $\mu_{a}\times \mu_{b}$ matrix ${}^{\mu}\widetilde E_{a,b}(u)$, respectively.
Then it is easy to obtain
	\begin{align}
		{}^{\mu}\widetilde E_{a,a}(u)&={}^{\mu}\widetilde F_{a,a}(u)=I_{\mu_a},\quad \mbox{for}\, 1{\leq}a{\leq}n,\\	{}^{\mu}\widetilde{E}_{a,b}(u)&=\sum_{a=i_0<i_1<\ldots<i_s=b}(-1)^s~{}^{\mu}E_{a,i_1}(u){}^{\mu}E_{i_1,i_2}(u)\cdots {}^{\mu}E_{i_{s-1},b}(u),\,\mbox{for}\, a<b,\label{einverse}\\
		{}^{\mu}\widetilde{F}_{b,a}(u)&=\sum_{a=i_0<i_1<\ldots<i_s=b}(-1)^s~{}^{\mu}F_{b,i_{s-1}}(u){}^{\mu}F_{i_{s-1},i_{s-2}}(u)\cdots {}^{\mu}F_{i_{1},a}(u).\label{finverse}	
	\end{align}
	For example,
	the relation \eqref{einverse} for $\widetilde{E}_{a,b}(u)$ is from
	the recursive relation
	\[{}^{\mu}\widetilde E_{a,b}(u)
	=-\big({}^{\mu}E_{a,b}(u)+{}^{\mu}\widetilde E_{a,a+1}(u){}^{\mu}E_{a+1,b}+\cdots +{}^{\mu}\widetilde E_{a,b-1}(u){}^{\mu}E_{b-1,b}(u)\big)\] 	
	produced from the equality $E(u)^{-1}E(u)=\textbf{I}$.

\begin{proposition}\label{swap}		
The map $\zeta_{M|N}:\Ymn(\so)\rightarrow Y_{N|M}(\sd)$ defined by the composition of isomorphisms $\omega_{M|N}$ and $\rho_{M|N}$
:$\zeta_{M|N}=\rho_{M|N}\circ\omega_{M|N}$ is an isomorphism with 
\begin{equation}\label{zeta}
\zeta_{M|N}\big(t_{ij}(u)\big)=t_{M+N+1-i,M+N+1-j}^{\prime}(u).
\end{equation}
for all admissible $i,j$.
Note that those on the right hand side are in $Y_{N|M}(\sd)=Y_{\overleftarrow{\mu}}(\sd)$.		
\end{proposition}
It can be seen that the map $\zeta_{M|N}$ turns the matrix $T(u)$ up side down, 
that is, for $1{\leq}a,b{\leq}n$,
each $(n-a,n-b)$-th block matrix in $Y_{N|M}(\sd)$ has the same size $\mu_a\times \mu_b$ with the $(a,b)$-th block in $\Ymn(\so)$.
It follows that we will consider the following reverse of $\mu$:
\[\overleftarrow{\mu}:=(\mu_n,\mu_{n-1},\ldots,\mu_{1}),\]
and take $\sd$ in the image space in order to keep the parities unchanged.

To avoid the abuse of notation,
we can define in the same way the elements \{$\overleftarrow{D}_{a;i,j}(u);\overleftarrow{D}^{\prime}_{a;i,j}(u)$\}, \{$\overleftarrow{E}_{a;i,j}(u)$\}, \{$\overleftarrow{F}_{a;i,j}(u)$\}, or \{$\overleftarrow{D}_{a;i,j}^{(r)};\overleftarrow{D}^{\prime(r)}_{a;i,j}$\}, \{$\overleftarrow{E}_{a;i,j}^{(r)}$\}, \{$\overleftarrow{F}_{a;i,j}^{(r)}$\} in $Y_{N|M}(\sd)=Y_{\overleftarrow{\mu}}(\sd)$ by Gauss decomposition with respect to $\overleftarrow{\mu}$,
and other corresponding elements.

Thus for any $1{\leq}i{\leq}\mu_a$ and $1{\leq}j{\leq}\mu_b$,
the map $\zeta_{M|N}$ sends the $(i,j)$-th entry $T_{a,b;i,j}=t_{\sum_{k=1}^{a-1}\mu_{k}+i,\sum_{k=1}^{b-1}\mu_{k}+j}$ in the $(a,b)$-th block of $T(u)$ in $\Ymn(\so)$, 
to the following entry  
$$\overleftarrow{t'}_{M{+}N{+}1{-}(\sum_{k=1}^{a-1}\mu_{k}{+}i),M{+}N{+}1{-}(\sum_{k=1}^{b-1}\mu_{k}{+}j)}$$ of the matrix $\overleftarrow{T'}(u):=T(u)^{-1}$ in 
$Y_{N|M}(\sd)$ located in $\mu_a{+}1{-}i$-row and $\mu_b{+}1{-}j$-column for the $(n{+}1{-}a,n{+}1{-}b)$-th block in $\overleftarrow{T'}(u)$ with respect to $\overleftarrow{\mu}$.
So the map $\zeta_{M|N}$ can be expressed as \[\zeta_{M|N}\big(t_{\sum_{k=1}^{a-1}\mu_{k}+i,\sum_{k=1}^{b-1}\mu_{k}+j}\big)=\overleftarrow{t'}_{M+N+1-(\sum_{k=1}^{a-1}\mu_{k}+i),M+N+1-(\sum_{k=1}^{b-1}\mu_{k}+j)},\]
which equals to
\begin{equation}\label{zetaprime}
\zeta_{M|N}\big(T_{a,b;i,j}(u)\big)=
\overleftarrow{T'}_{n+1-a,n+1-b;\mu_a+1-i,\mu_b+1-j}.
\end{equation}
To be precise,
we have
\begin{proposition}
For any shape $\mu$ and $\so$,
we have
\begin{eqnarray}
	\zeta_{M|N}\big(D_{a;i,j}(u)\big)&=&\overleftarrow{D}^{\prime}_{n+1-a;\mu_{a}+1-i,\mu_{a}+1-j}(u),\, \forall 1{\leq}a{\leq}n,\label{zetad}	\\
	\zeta_{M|N}\big(E_{a,b;i,j}(u)\big)&=&-\overleftarrow{F}_{n+1-a,n+1-b;\mu_a+1-i,\mu_b+1-j}(u),\,\forall 1{\leq}a<b{\leq}n\label{zetae},\\
	\zeta_{M|N}\big(F_{b,a;i,j}(u)\big)&=&-\overleftarrow{E}_{n+1-b,n+1-a;\mu_b+1-i,\mu_a+1-j}(u),\,\forall 1{\leq}a<b{\leq}n.\label{zetaf}
\end{eqnarray}
\end{proposition}
\begin{proof}
Multiplying  out the matrix products
\[ T(u)=F(u)D(u)E(u) \qquad \text{and} \qquad T(u)^{-1}=E(u)^{-1}D(u)^{\prime}F(u)^{-1},\]
we obtin the following matrix identities for any given composition $\mu$
\begin{eqnarray}
	{}^{\mu}T_{a,a}(u)&=&{}^{\mu}D_a(u)+\sum_{c<a}{}^{\mu}F_{a,c}(u){}^{\mu}D_c(u){}^{\mu}E_{c,a}(u),\label{td}\\
	{}^{\mu}T_{a,b}(u)&=&{}^{\mu}D_a(u){}^{\mu}E_{a,b}(u)+\sum_{c<a}{}^{\mu}F_{a,c}(u){}^{\mu}D_c(u){}^{\mu}E_{c,b}(u),\label{tu}\\
	{}^{\mu}T_{b,a}(u)&=&{}^{\mu}F_{b,a}(u){}^{\mu}D_a(u)+\sum_{c<a}{}^{\mu}F_{b,c}(u){}^{\mu}D_c(u){}^{\mu}E_{c,a}(u),\label{tl}
\end{eqnarray}
and
\begin{eqnarray}
	{}^{\mu}T^{\prime}_{a,a}(u)&=&{}^{\mu}D^{\prime}_a(u)+\sum_{c>a}{}^{\mu}\widetilde{E}_{a,c}(u){}^{\mu}D^{\prime}_c(u){}^{\mu}\widetilde{F}_{c,a}(u),\label{primed}\\ {}^{\mu}T_{a,b}^{\prime}(u)&=&{}^{\mu}\widetilde E_{a,b}(u){}^{\mu}D^{\prime}_b(u)+\sum_{c>b}{}^{\mu}\widetilde E_{a,c}(u){}^{\mu}D^{\prime}_c(u){}^{\mu}\widetilde F_{c,b}(u),
	\label{primeu}\\
	{}^{\mu}T_{b,a}^{\prime}(u)&=&{}^{\mu}D^{\prime}_b(u){}^{\mu}\widetilde F_{b,a}(u)+\sum_{c>b}{}^{\mu}\widetilde E_{b,c}(u){}^{\mu}D^{\prime}_c(u){}^{\mu}\widetilde F_{c,a}(u),\label{primel}
\end{eqnarray} 
for all $1{\leq}a{\leq}n$  and $1{\leq}a<b{\leq}n$. 
Here,
${}^{\mu}T_{a,b}^{\prime}(u)$ denotes the $\mu_a\times\mu_b$-matrices in the $(a,b)$-th block position of $T(u)^{-1}$.

Applying the map $\zeta_{M|N}$ to $^{\mu}D_1(u)={}^{\mu}T_{1,1}(u)$ obtained by the special case  when $a=1$ in \eqref{td},
it gives
\begin{align*}
&\zeta_{M|N}(D_{1;i,j}(u))\\
=&\zeta_{M|N}(t_{i,j})=\overleftarrow{t'}_{M+N+1-i,M+N+1-j}(u)=\overleftarrow{T'}_{n,n;\mu_1+1-i,\mu_1+1-j}\\
=&_{\eqref{primed}}\overleftarrow{D'}_{n,n;\mu_1+1-i,\mu_1+1-j}(u)=\overleftarrow{D'}_{n;\mu_1+1-i,\mu_1+1-j}(u),
\quad \forall~ 1{\leq}i,j{\leq}\mu_{1}.
\end{align*}

Also,
applying $\zeta_{M|N}$ to $ ^{\mu}D_1(u)\cdot {}^{\mu}E_{1,b}(u)={}^{\mu}T_{1,b}(u)$,
$^{\mu}F_{b,1}(u)\cdot {}^{\mu}D_1(u)={}^{\mu}T_{b,1}(u)$ given by the special case  when $a=1$ in \eqref{tu}-\eqref{tl},
we have
\begin{align*}
\sum_{k}\overleftarrow{t'}_{M+N+1-i,M+N+1-k}(u)\cdot \zeta_{M|N}(E_{1,b;k,j}(u))=\overleftarrow{t'}_{M+N+1-i,M+N+1-(\mu_1+\cdots+\mu_{b-1}+j)}(u),\\
		\sum_{k}\zeta_{M|N}(F_{b,1;j,k}(u))\cdot \overleftarrow{t'}_{M+N+1-k,M+N+1-i}(u)=\overleftarrow{t'}_{M+N+1-(\mu_1+\cdots+\mu_{b-1}+j,M+N+1-i)}(u),
\end{align*}
which equals to
\begin{align*}
	\sum_{k}\overleftarrow{D'}_{n;\mu_1+1-i,\mu_1+1-k}(u)\cdot \zeta_{M|N}(E_{1,b;k,j}(u))=\overleftarrow{T'}_{n,n+1-b;\mu_1+1-i,\mu_b+1-j},\\
	\sum_{k}\zeta_{M|N}(F_{b,1;j,k}(u))\cdot \overleftarrow{D'}_{n;\mu_1+1-k,\mu_1+1-i}(u)=\overleftarrow{T'}_{n+1-b,n;\mu_b+1-j,\mu_1+1-i},
\end{align*}
for $1{\leq}i{\leq}\mu_1$ and $1{\leq}j{\leq}\mu_b$.
Associated with ${}^{\overleftarrow{\mu}}T_{a,b}^{\prime}(u)={}^{\overleftarrow{\mu}}\widetilde E_{a,b}(u){}^{\overleftarrow{\mu}}D^{\prime}_b(u)$, and ${}^{\overleftarrow{\mu}}T_{b,a}^{\prime}(u)={}^{\overleftarrow{\mu}}D^{\prime}_b(u){}^{\mu}\widetilde F_{b,a}(u)$ obtained by the case of when $a=n$ in \eqref{primeu}-\eqref{primel} in $Y_{N|M}(\sd)$, respectively,
then we have immediately that the cases when $a=1$ of the above more general relations \eqref{zetad}--\eqref{zetaf},
which can be easily derived simultaneously by induction on $a$ in a simlilar way with the case of $a=1$.
\end{proof}

Associated with \eqref{einverse} and \eqref{finverse}, 
their relations with the \textit{parabolic generators} \{$D_{a;i,j}^{(r)};D^{\prime(r)}_{a;i,j}$\}, \{$E_{a;i,j}^{(r)}$\}, \{$F_{a;i,j}^{(r)}$\} in $\Ymn(\so)=Y_{\mu}(\so)$ under the map $\zeta_{M|N}$
can be obtained directly by the special cases when $b=a+1$ in proposition \ref{swap},
which is given the following corollary \ref{zetac}.
\begin{corollary}\label{zetac}
For any fixed $\so$ and the shape $\mu$,
we have\begin{eqnarray}
	\zeta_{M|N}\big(D_{a;i,j}(u)\big)&=&\overleftarrow{D}^{\prime}_{n+1-a;\mu_{a}+1-i,\mu_{a}+1-j}(u), \qquad\forall 1{\leq}a{\leq}n, \label{zd} \\
	\zeta_{M|N}\big(E_{a;i,j}(u)\big)&=&-\overleftarrow{F}_{n-a;\mu_{a}+1-i,\mu_{a+1}+1-j}(u), \quad\forall 1{\leq}a{\leq}n-1, \label{sze}\\
	\zeta_{M|N}\big(F_{a;j,i}(u)\big)&=&-\overleftarrow{E}_{n-a;\mu_{a+1}+1-j,\mu_{a}+1-i}(u), \quad\forall 1{\leq}a{\leq}n-1\label{szf}.
\end{eqnarray}
for all admissible $i,j$.
Note that those on the right hand side are in $Y_{N|M}(\sd)=Y_{\overleftarrow{\mu}}(\sd)$.
\end{corollary}

Next,
we analyze the \textit{shift map} $\psi_{p|q}$ for any $p,q\in\mathbb{Z}_{\ge0}$.
Let $\so_1$ be an arbitrary $0^p1^q$-sequence.
Starting with any shape $\mu$ and the sequence $\so$,
we have 
\begin{proposition}\label{shift}		
Let $\varphi_{p|q}:Y_{M|N}(\so)\rightarrow Y_{p+M|q+N}(\so_1\so)$
be the injective homomorphism sending $t_{i,j}^{(r)}$ to $t_{p+q+i,p+q+j}^{(r)}$. 
The map $\psi_{p|q}:\Ymn(\so)\rightarrow Y_{p+M|q+N}(\so_1\so)$ defined by
		\[
		\psi_{p|q}=\omega_{p+M|q+N}\circ\varphi_{p|q}\circ\omega_{M|N},
		\]
		is an injective homomorphism, 
		and for any $1{\leq}i,j{\leq}M+N$
\begin{align}
	\psi_{p|q}\big(t_{ij}(u)\big)&=
	\left| \begin{array}{cccc} t_{1,1}(u) &\cdots &t_{1,p+q}(u) &t_{1, p+q+j}(u)\\
		\vdots &\ddots &\vdots &\vdots \\
		t_{p+q,1}(u) &\cdots &t_{p+q,p+q}(u) &t_{p+q, p+q+j}(u)\\
		t_{p+q+i, 1}(u) &\cdots &t_{p+q+i,p+q}(u) &\boxed{t_{p+q+i, p+q+j}(u)}
	\end{array} \right|.\label{pqmap}
\end{align}
\end{proposition}
\begin{proof}
We will demonstrate that the proof in \cite{BK1} also works perfectly for the modular super case.
Let $T(u)$ denote the matrix $\left(t_{i,j(u)}\right)_{1{\leq}i,j{\leq}M+N+p+q}$ with entries in the modular super Yangian $Y_{p+M|q+N}(\so_1\so)[[u^{-1}]]$ as usual,
and let $\widetilde{T}(-u):=T(-u)^{-1}$.
We write the matrix $\widetilde{T}(-u), T(-u)$ in the following block form
\begin{equation}\label{blockpq}
	\left(\begin{array}{cc}
	\widetilde{A}(-u)&\widetilde{B}(-u)\\\widetilde{C}(-u)&\widetilde{D}(-u)
\end{array}\right):=\left(\begin{array}{cc}
A(-u)&B(-u)\\C(-u)&D(-u)
\end{array}\right)^{-1}
\end{equation}
where $A(-u), B(-u), C(-u), D(-u)$ are 
$(p+q)\times (p+q)$, $(p+q)\times (M+N)$,
$(M+N)\times (p+q)$, $(M+N)\times (M+N)$
matrices,
respectively.
By block multiplication and the fact that $A(-u), D(-u)$ are inverse,
it is easy to obtain the following classical identities
\begin{align*}
\widetilde{A}(-u)&=\left(A(-u)-B(-u)D(-u)^{-1}C(-u)\right)^{-1},\\
\widetilde{B}(-u)&=-A(-u)^{-1}B(-u)\left(D(-u)-C(-u)A(-u)^{-1}B(-u)\right)^{-1},\\
\widetilde{C}(-u)&=-D(-u)^{-1}C(-u)\left(A(-u)-B(-u)D(-u)^{-1}C(-u)\right)^{-1},\\
\widetilde{D}(-u)&=\left(D(-u)-C(-u)A(-u)^{-1}B(-u)\right)^{-1}.
\end{align*}
Now,
by the definition of the map $\psi_{p|q}$,
we have
\begin{equation}\label{pqmap'}
\begin{array}{cccccccc}
\psi_{p|q}:&\Ymn(\so)&\rightarrow&\Ymn(\so)&\rightarrow& Y_{p+M|q+N}(\so_1\so)&\rightarrow&Y_{p+M|q+N}(\so_1\so)\\
&T(u)&\rightarrow&T(-u)^{-1}&\rightarrow&\widetilde{D}(-u)&\rightarrow&\widetilde{D}(u)^{-1}
\end{array}
\end{equation}
Thus it maps the $(M+N)\times(M+N)$ matrix $T(u)$ of $\Ymn(\so)$ to the $(M+N)\times(M+N)$ matrix $D(u)-C(u)A(u)^{-1}B(u)$ in $Y_{p+M|q+N}(\so_1\so)$.
Then the proposition follows from this on the following computation of $ij$-erntries:
\begin{align*}
&\psi_{p|q}(t_{ij}(u))=\left(D(u)-C(u)A(u)^{-1}B(u)\right)_{ij}=\left(D(u)\right)_{ij}-\left(C(u)A(u)^{-1}B(u)\right)_{ij}\\
=&t_{p+q+i,p+q+j}(u)-\sum\limits_{k,l=1}^{p+q}t_{p+q+i,k}(u)A(u)^{-1}_{k,l}t_{l,p+q+j}(u)\\
=&t_{p+q+i,p+q+j}(u){-}\left(t_{p+q+i,1}(u),\cdots,t_{p+q+i,p+q}(u)\right)A(u)^{{-}1}\left(t_{1,p+q+j}(u),\cdots,t_{p+q,p+q+j}(u)\right)^{t}\\
=&\mbox{the~quasideterminant~of ~the~right~side~of~}\eqref{pqmap}.
\end{align*}
\end{proof}

The relations in \eqref{pqmap'} tell us that the \textit{shift map} $\psi_{p|q}$ embeds the matrix $T(u)$ of $\Ymn(\so)$ into the southeastern corner of  the larger matrix of $Y_{p+M|q+N}(\so_1\so)$. 
Also,
the image of the map $\psi_{p|q}$ in \eqref{pqmap} means that it depends only on $p+q$,
so we may simply write $\psi_{p|q}=\psi_L$ for $L=p+q$ when appropriate.
\begin{corollary}
Under any composition $\mu=(\mu_1,\cdots,\mu_n)$ of $M+N$,
then for any $a$ with $2{\leq}a{\leq}n$ and $L=\mu_1+\cdots+\mu_{a-1}$,
we have
\begin{equation}\label{pqele}
\psi_L(D_{1;i,j}(u))=D_{a;i,j}(u),
\psi_L(E_{1;i,j}(u))=E_{a;i,j}(u),
\psi_L(F_{1;i,j}(u))=F_{a;i,j}(u),
\end{equation}
which is also equivalent to
 \begin{equation}\label{pqele1}
 	\psi_L(^{\mu}D_{1}(u))={}^{\mu}D_{a}(u),\quad
 	\psi_L(^{\mu}E_{1}(u))={}^{\mu}E_{a}(u),\quad
 	\psi_L(^{\mu}F_{1}(u))={}^{\mu}F_{a}(u).
 \end{equation}
\end{corollary}
\begin{proof}
Taking $D_{a}(u)$ as an example to prove it step by step.
\begin{align*}
&\psi_L(D_{1;i,j}(u))=\psi_L({}^{\mu}T_{1,1}(u)_{i,j})=\psi_L(T_{1,1;i,j}(u)=\psi_L(t_{i,j}(u))\\
=&_{\eqref{pqmap}}t_{L+i,L+j}(u){-}\left(t_{L+i,1}(u),\cdots,t_{L+i,L}(u)\right)\left(\begin{array}{ccc}
 t_{1,1}(u) &\cdots &t_{1,L}(u)\\
\vdots &\ddots &\vdots \\
t_{L,1}(u) &\cdots &t_{L,L}(u)
\end{array}\right)^{-1}
\left(\begin{array}{c}
t_{1,L+j}\\\vdots\\t_{L,L+j}
\end{array}\right)
\\
=&T_{a,a;i,j}(u){-}\left(T_{a,1;i,1}(u),\cdots,T_{a,1;i,L}(u)\right)\left(
\begin{array}{ccc}
	{^\mu}T_{1,1}(u) & \cdots &{^\mu}T_{1,a-1}(u)\\
	\vdots & \ddots &\vdots\\
	{^\mu}T_{a-1,1}(u) & \cdots & {^\mu}T_{a-1,a-1}(u)
\end{array}
\right)^{-1}
\left(\begin{array}{c}
	T_{1,a;1,j}\\\vdots\\T_{1,a;L,j}
\end{array}\right)\\
=&\mbox{the~}ij-\mbox{entries~of~the~following~matrix}\\
&{}^{\mu}T_{a,a}(u){-}\left({^\mu}T_{a,1}(u), \cdots, {^\mu}T_{a,a-1}(u)\right)\left(
\begin{array}{ccc}
	{^\mu}T_{1,1}(u) & \cdots &{^\mu}T_{1,a-1}(u)\\
	\vdots & \ddots &\vdots\\
	{^\mu}T_{a-1,1}(u) & \cdots & {^\mu}T_{a-1,a-1}(u)
\end{array}
\right)^{{-}1}
\left(\begin{array}{c}
	{^\mu}T_{1,a}(u)\\\vdots\\{^\mu}T_{a-1,a}(u)
\end{array}\right)\\
=&_{\eqref{quasid}}D_{a;i,j}(u).
\end{align*}
 $E_{a}(u)$, $F_{a}(u)$ can be obtained similarly by \eqref{quasie}, \eqref{quasif}, respectively.
\end{proof}

Since $\omega_{M|N}$ is an involution,
and then it maps $T(u)^{-1}$ to $T(-u)$,
then associated to the proof of Proposition \ref{shift},
it follows that the map $\psi_L$ sends the element $t_{i,j}^{\prime}(u)$ in $\Ymn(\so)$ to $t_{L+i,L+j}^{\prime}(u)$ in $Y_{p+M|q+N}(\so_1\so)$. 
It means that the image of $\psi_{L}\big(Y_{M|N}(\so)\big)$ in $Y_{p+M|q+N}(\so_1\so)$ is generated by the following elements 
\[
\left\{t_{L+i,L+j}^{\prime(r)}\in Y_{p+M|q+N}(\so_1\so)\,|\,1{\leq}i,j{\leq}M+N, r\geq 0\right\}.
\] 

The block form \eqref{blockpq} means that if we pick any element $t_{i,j}^{(r)}$ in $A(u)$, the northwestern $(p+q)\times (p+q)$ corner of $T(u)$ in $Y_{p+M|q+N}[[u^{-1}]]$, then their indices will never overlap with those of $\psi_{L}\big(Y_{M|N}\big)$, that are in the $\widetilde{D}(u)$, southeastern $(M+N)\times(M+N)$ corner of $T(u)^{-1}$. As a result of equation (\ref{commurelation}), they supercommute.
Obviously,
the elements in the  $(p+q)\times (p+q)$ corner $A(u)$ generate a sualgebra which is isomorphic to $Y_{p|q}(\so_1)$.
\begin{corollary}\label{cor:commu}
	For any $p,q\in\mathbb{Z}_{\ge0}$
and an arbitrary $0^p1^q$-sequence $\so_1$,
we have the subalgebras $Y_{p|q}(\so_1)$ and $\psi_{L}\big(Y_{M|N}(\so)\big)$ supercommute with each other, in the Yangian $Y_{p+M|q+N}(\so_1\so)$,
where $L=p+q$.
\end{corollary}

Thence,
$D_{a;i,j}$ and $D_{b;h,k}$ are super-commutative when $|a-b|>0$.
Also the following relations hold if $|a-b|>1$
$$[E_{a;i,j}(u),E_{b;h,k}(v)]=0,\, |a-b|>1.$$

Moreover,
by Eqs. \eqref{pqmap} and \eqref{pqele},
the parities of the parabolic generators are obtained as follows:
\begin{eqnarray}
	\label{pad}\text{ parity of } D_{a;i,j}^{(r)}&=&|i|_a+|j|_a \,\,\text{(mod 2)},\\
	\label{pae}\text{ parity of } E_{b;h,k}^{(s)}&=&|h|_{b}+|k|_{b+1} \,\,\text{(mod 2)},\\
	\label{paf}\text{ parity of } F_{b;l,t}^{(s)}&=&|l|_{b+1}+|t|_{b} \,\,\text{(mod 2)}.
\end{eqnarray}

\begin{proposition}\label{dd0}
The following super-commutative relations hold:
\begin{align*}
[D_{a;i,j}(u),E_{b;h,k}(v)]=0,\, [D_{a;i,j}(u),F_{b;h,k}(v)]=0,\,&\mbox{if}~ b-a\geq 1, \mbox{or}~a-b>1;\\
[F_{a;i,j}(u),F_{b;h,k}(v)]=0,\, [E_{a;i,j}(u),E_{b;h,k}(v)]=0,\, &\mbox{if}~ |a-b|>1.
\end{align*}
Also,  
the relations among the elements
$\{D_{a;i,j}^{(r)},D_{a;i,j}^{\prime (r)}\}$ for all $r\geq 0$, ${1{\leq}i,j{\leq}\mu_a}$, $1{\leq}a{\leq}n$
are given by
\[
D_{a;i,j}^{(0)}=\delta_{ij},\]
\[
{\displaystyle \sum_{t=0}^{r}D_{a;i,p}^{(t)}D_{a;p,j}^{\prime (r-t)}=\delta_{r0}\delta_{ij} },
\]
\begin{align*}
	[D_{a;i,j}^{(r)},D_{b;h,k}^{(s)}]=&
	\delta_{ab}(-1)^{|i|_a|j|_a+|i|_a|h|_a+|j|_a|h|_a}\\
	&\times\sum_{t=0}^{min(r,s)-1}\big(D_{a;h,j}^{(t)}D_{a;i,k}^{(r+s-1-t)}
	-D_{a;h,j}^{(r+s-1-t)}D_{a;i,k}^{(t)}\big).\\
\end{align*}	
\end{proposition}
\begin{proof}
We only need to consider the case that $a=b$  by Corollary \ref{cor:commu},
and then these relations follows directly from $a=1$ since the map $\psi_{L}$ and \eqref{RTT relations}.
\end{proof}
In the case of \cite{BK1} and \cite{Peng16},
if the generators are from two different blocks and the blocks are not ``close", then they commute.
Fortunately,  Corollary \ref{cor:commu} and Prposition \ref{dd0} mean that this phenomenon remains to be true in characteristic $p$.
As a consequence, we only have to focus on the supercommutation relations of the elements that are either in the same block, or their belonging blocks are ``close enough" for the length of $\mu$ is $2,3,4$,
the situations are less complicated so that we may derive various relations among those generators by direct computation.

Then we will take advantage of the above homomorphisms $\psi_L$ and $\zeta_{M|N}$ between super Yangians since these maps carry the relations in the special cases that $n{\leq}4$ to the general case, so that we obtain many relations in $Y_{M|N}$. 
Finally we prove that we have found enough relations for our presentation.

\section{Special Case: $n=2$}
In this section, we focus on the very first non-trivial case under our consideration; that is, $\mu=(\mu_1,\mu_2)$ with a fixed 01-sequence $\so=\so_1\so_2$. 
We list our parabolic generators as follow:
\begin{align*}
	&\big\lbrace D_{a;i,j}^{(r)}, D_{a;i,j}^{\prime(r)} \,|\, {a=1,2;\; 1{\leq}i,j{\leq}\mu_a;\; r\geq 0}\big\rbrace,\\
	&\big\lbrace E_{1;i,j}^{(r)} \,|\, 1{\leq}i{\leq}\mu_1, 1{\leq}j\leq\mu_{2};\; r\geq 1\big\rbrace,\\
	&\big\lbrace F_{1;i,j}^{(r)} \,|\, 1{\leq}i\leq\mu_{2}, 1{\leq}j{\leq}\mu_1;\; r\geq 1\big\rbrace.
\end{align*}

We compute the matrix products \eqref{Tp=FDE} and \eqref{gaussdec} with respect to the composition $\mu=(\mu_1,\mu_2)$ as follows,
\begin{align*}
	{^\mu}T(u)&
	=\left(
	\begin{array}{ccc}
		D_1(u) &D_1(u)E_1(u)\\
		F_1(u)D_1(u)&F_1(u)D_1(u)E_1(u)+D_2(u)
	\end{array}
	\right),\\
	{^\mu}T(u)^{-1}&=\left(\begin{array}{cc}
		D_1(u)^{-1}+E_1(u)D_2(u)^{-1}F_1(u)&-E_1(v)D_2(u)^{-1}\\
		-D_2(u)^{-1}F_{1}(u)&D_2(u)^{-1}
	\end{array}\right),
\end{align*}
and then get the following identities.	
\begin{equation}\label{eq:n=2}
	\begin{aligned}
		t_{i,j}(u)&=D_{1;i,j}(u), &\forall 1\le i,j\le \mu_1,\\
		t_{i,\mu_1+j}(u)&=D_{1;i,p}E_{1;p,j}(u),&\forall 1{\le} i{\le} \mu_1, 1{\le} j{\le} \mu_2,\\
		t_{\mu_1+i,j}(u)&=F_{1;i,p}(u)D_{1;p,j}(u),&\forall 1{\le} i{\le} \mu_2, 1{\le} j{\le} \mu_1,\\
		t_{\mu_1+i,\mu_1+j}(u)&=F_{1;i,p}(u)D_{1;p,q}(u)E_{1;q,j}(u){+}D_{2;i,j}(u),&\forall 1\le i,j\le \mu_2,\\
		t^{\prime}_{i,j}(u)&=D^{\prime}_{1;i,j}(u){+}E_{1;i,p'}(u)D^{\prime}_{2;p',q'}(u)F_{1;q',j}(u),&\forall 1{\le} i,j{\le} \mu_1,\\
		t^{\prime}_{i,\mu_1+j}(u)&=-E_{1;i,p'}(u)D^{\prime}_{2;p',j}(u),&\forall 1{\le} i{\le} \mu_1, 1{\le} j{\le} \mu_2,\\
		t^{\prime}_{\mu_1+i,j}(u)&=-D^{\prime}_{2;i,p'}(u)F_{1;p',j}(u),&\forall 1{\le} i{\le} \mu_2, 1{\le} j{\le} \mu_1,\\
		t^{\prime}_{\mu_1+i,\mu_1+j}(u)&=D^{\prime}_{2;i,j}(u),&\forall 1\le i,j\le \mu_2,
	\end{aligned}
\end{equation}
where the indices $p,q$ (respectively, $p',q'$) are summed over $1,\ldots,\mu_1$ (respectively, $1,\ldots,\mu_2$).

The following proposition gives explicitly the relations among the generators other than those relations already obtained in Proposition \ref{dd0}.
{\allowdisplaybreaks
	\begin{proposition}\label{n=2}
		Let $\mu=(\mu_1,\mu_2)$ be a composition of $M+N$. The following identities hold in
		$Y_{\mu}((u^{-1},v^{-1}))$:
		\begin{align}
			(u-v)[D_{1;i,j}(u), E_{1;h,k}(v)]\label{d1e1}
			&=(-1)^{|h|_1|j|_1}\delta_{hj}\sum_{p=1}^{\mu_1}D_{1;i,p}(u)\big(E_{1;p,k}(v)-E_{1;p,k}(u)\big),\\[2mm]
			(u-v)[E_{1;i,j}(u), D^{\prime}_{2;h,k}(v)]
			&=(-1)^{|h|_2|j|_2}\del_{hj} \sum_{q=1}^{\mu_2}\big(E_{1;i,q}(u)-E_{1;i,q}(v)\big)D^{\prime}_{2;q,k}(v),\label{e1d2}\\[2mm]        
			(u-v)[D_{2;i,j}(u), E_{1;h,k}(v)]
			&\notag=(-1)^{|h|_1|k|_2+|h|_1|j|_2+|j|_2|k|_2}\\
			&\qquad \times D_{2;i,k}(u)\big(E_{1;h,j}(u)-E_{1;h,j}(v)\big),\label{d2e1}\\[2mm]              
			(u-v)[D_{1;i,j}(u), F_{1;h,k}(v)]
			&\notag=(-1)^{|i|_1|j|_1+|h|_2|i|_1+|h|_2|j|_1}\delta_{ik}\\
			&\qquad \times\sum_{p=1}^{\mu_1}\big(F_{1;h,p}(u)-F_{1;h,p}(v)\big)D_{1;p,j}(u),\label{d1f1}\\[2mm]
			(u-v)[F_{1;i,j}(u), D^{\prime}_{2;h,k}(v)]
			&\notag=(-1)^{|h|_2|i|_2+|h|_2|j|_1+|j|_1|k|_2}\del_{ik}\\
			&\qquad \times\sum_{q=1}^{\mu_2}D^{\prime}_{2;h,q}(v)\big(F_{1;q,j}(v)-F_{1;q,j}(u)\big),\label{f1d2}\\[2mm]          
			(u-v)[D_{2;i,j}(u), F_{1;h,k}(v)]
			&\notag=(-1)^{|h|_2|k|_1+|h|_2|j|_2+|j|_2|k|_1}\\
			&\qquad\times\big(F_{1;i,k}(v)-F_{1;i,k}(u)\big)D_{2;h,j}(u),\label{d2f1}\\[2mm]       
			(u-v)[E_{1;i,j}(u), F_{1;h,k}(v)]
			&\notag=(-1)^{|h|_2|i|_1+|i|_1|j|_2+|h|_2|j|_2}D_{2;h,j}(u) D^{\prime}_{1;i,k}(u)\\
			&\qquad -(-1)^{|h|_2|k|_1+|j|_2|k|_1+|h|_2|j|_2}
			D^{\prime}_{1;i,k}(v) D_{2;h,j}(v),\label{e1f1}\\[2mm]
		(u-v)[E_{1;i,j}(u), E_{1;h,k}(v)]
			&\notag=(-1)^{|h|_1|j|_2+|j|_2|k|_2+|h|_1|k|_2}\\
			&\qquad \times\big(E_{1;i,k}(u)-E_{1;i,k}(v)\big)\big(E_{1;h,j}(u)-E_{1;h,j}(v)\big),\label{e1e1}\\[2mm]
		(u-v)[F_{1;i,j}(u), F_{1;h,k}(v)]
			&\notag=-(-1)^{|i|_2|j|_1+|h|_2|i|_2+|h|_2|j|_1}\\
			&\qquad \times\big(F_{1;h,j}(u)-F_{1;h,j}(v)\big)\big(F_{1;i,k}(u)-F_{1;i,k}(v)\big).\label{f1f1} 
		\end{align}
		The identities hold for all $1{\leq}i,j{\leq}\mu_1$ if $D_{1;i.j}(u)$ appears on the left-hand side of the equation, 
		for all $1{\leq}h,k{\leq}\mu_2$ if $D_{2;h,k}^\prime(u)$ appears on the left-hand side of the equation, 
		for all $1{\leq}i^\prime {\leq}\mu_1$, $1{\leq}j^\prime {\leq}\mu_2$ if $E_{1;i^\prime,j^\prime}(u)$ appears on the left-hand side of the equation, and
		for all $1{\leq}h^\prime{\leq}\mu_2$, $1{\leq}k^\prime{\leq}\mu_1$ if $F_{1;h^\prime,k^\prime}(u)$ appears on the left-hand side of the equation.         
	\end{proposition}
}

\begin{remark}\label{re:FF}
Together with the identity \begin{align}\label{gv-gu/u-v}
		\frac{g(v)-g(u)}{u-v}=\sum\limits_{r,s\geq 1}g^{(r+s-1)}u^{-r}v^{-s}
	\end{align}
	for any formal series $g(u)=\sum_{r\geq 0}g^{(r)}u^{-r}$, 
	we have 
	\begin{equation}\label{FF}
		[E_{1;i,j}(v),E_{1;h,k}(v)]=[F_{1;i,j}(v),F_{1;h,k}(v)]=0
	\end{equation}
	obtained by dividing two side of \eqref{e1e1} and \eqref{f1f1} by $(u-v)$ and setting $u=v$, respectively.
\end{remark}

\begin{proof}
Equations  \eqref{d1e1}--\eqref{d2e1} and \eqref{e1e1} follow from applying the map $\zeta_{M|N}$ to \eqref{d1f1}--\eqref{d2f1} and \eqref{f1f1} just with suitable choices of indices.
So we will prove \eqref{f1f1} in detail here for $F$'s as an illustrating example  to demonstrate that the proof in \cite[Lemma 6.3]{BK1} and \cite[Section 5]{Peng16} mainly related to $E'$s also works in positive characteristic.

To do that,
we will start with the relation $(u-v)^2[t_{\mu_1+i,j}(v),t'_{\mu_1+h,k}(u)]=0$ for all $1{\le} i{\le} \mu_2, 1{\le} j{\le} \mu_1$ and $1{\le} h{\le} \mu_2, 1{\le} k{\le} \mu_1$.
Then computing the bracket after substituting by the third identity and the seventh identity in \eqref{eq:n=2}, we have
\begin{equation}\label{f1-1}
\begin{aligned}
	&(u-v)^2D'_{2;h,q}(u)F_{1;q,k}(u)F_{1;i,p}(v)D_{1;p,j}(v)\\
	=&(-1)^{\left(|j|_1+|i|_2\right)\left(|k|_1+|h|_2\right)}(u-v)^2F_{1;i,p}(v)D_{1;p,j}(v)D'_{2;h,q}(u)F_{1;q,k}(u)\\
	=&(-1)^{\left(|j|_1{+}|i|_2\right)\left(|k|_1{+}|h|_2\right)}(-1)^{\left(|p|_1{+}|j|_1\right)\left(|h|_2{+}|q|_2\right)}(u{-}v)^2F_{1;i,p}(v)D'_{2;h,q}(u)D_{1;p,j}(v)F_{1;q,k}(u),
\end{aligned}
\end{equation}
where the indices $p$ and $q$ are summed from 1 to $\mu_1$ and $\mu_2$, respectively.
Then we compute the brackets in \eqref{d1f1} and \eqref{f1d2}, and obtain the following two identites:
\begin{multline*}
	(v-u)F_{1;i,p}(v)D'_{2;h,q}(u)=(v-u)(-1)^{\left(|p|_1+|i|_2\right)\left(|q|_2+|h|_2\right)}D'_{2;h,q}(u)F_{1;i,p}(v)\\
	+\delta_{iq}(-1)^{|i|_2|h|_2+|p|_1|h|_2+|p|_1|q|_2}D'_{2;h,g_1}(u)\left(F_{1;g_1,p}(u)-F_{1;g_1,p}(v)\right),
\end{multline*}
\begin{multline*}
	(v-u)D_{1;p,j}(v)F_{1;q,k}(u)=(v-u)(-1)^{\left(|p|_1+|j|_1\right)\left(|k|_1+|q|_2\right)}F_{1;q,k}(u)D_{1;p,j}(v)\\
	+\delta_{pk}(-1)^{|p|_1|j|_1+|p|_1|q|_2+|j|_1|q|_2}\left(F_{1;q,g_2}(v)-F_{1;q,g_2}(u)\right)D_{1;g_2,j}(v),
\end{multline*}
where the indices $g_1$ and $g_2$ are summed from 1 to $\mu_2$ and $\mu_1$, respectively.
Substituting these two identities into the right side of \eqref{f1-1},
then we derive
\begin{equation}\label{f1-3}
\begin{aligned}
&(u-v)^2D'_{2;h,q}(u)[F_{1;q,k}(u),F_{1;i,p}(v)]D_{1;p,j}(v)\\
=&(-1)^{|k|_1|i|_2+|k|_1|q|_2+|i|_2|q|_2}(v{-}u)D'_{2;h,q}(u)F_{1;i,k}(v)\left(F_{1;q,p}(v){-}F_{1;q,p}(u)\right) D_{1;p,j}(v)\\
&+(-1)^{|p|_1|k|_1+|p|_1|i|_2+|k|_1|i|_2}(v{-}u)D'_{2;h,q}(u)\left(F_{1;q,p}(u){-}F_{1;q,p}(v)\right)F_{1;i,k}(u) D_{1;p,j}(v)\\
&+D'_{2;h,q}(u)\left(F_{1;q,k}(u)-F_{1;q,k}(v)\right)\left(F_{1;i,p}(v)-F_{1;i,p}(u)\right) D_{1;p,j}(v)
\end{aligned}
\end{equation}where the indices $p$ and $q$ are summed from 1 to $\mu_1$ and $\mu_2$, respectively.

The indices $h$ and $j$ are not involved in the above those sign factors,
so multiplying each term in \eqref{f1-3} by $D_{2}(u)$ on the left, and by $D'_{1}(v)$ on the right, respectively,
we derive that
\begin{equation}\label{f1-3'}
	\begin{aligned}
		&(u-v)^2[F_{1;i,j}(u),F_{1;k,l}(v)]\\
		=&(-1)^{|j|_1|h|_2+|j|_1|i|_2+|h|_2|i|_2}(v{-}u)F_{1;h,j}(v)\left(F_{1;i,k}(v){-}F_{1;i,k}(u)\right)\\
		&+(-1)^{|k|_1|j|_1+|k|_1|h|_2+|j|_1|h|_2}(v{-}u)\left(F_{1;i,k}(u){-}F_{1;i,k}(v)\right)F_{1;h,j}(u)\\
		&+\left(F_{1;i,j}(u)-F_{1;i,j}(v)\right)\left(F_{1;h,k}(v)-F_{1;h,k}(u)\right).
	\end{aligned}
\end{equation}

For a power series $P$ in $Y_{\mu}[[u^{-1}, v^{-1}]]$, we write $\big\lbrace P\big\rbrace_{d}$ for the homogeneous component of $P$ of total degree $d$ in the variables $u^{-1}$ and $v^{-1}$,
then \eqref{f1f1} is a consequence of the following \eqref{f1-0}.
\begin{equation}\label{f1-0}
	\begin{aligned}	
		&(u-v)\left\{[F_{1;i,j}(u), F_{1;h,k}(v)]\right\}_{d+1}\\
		&=(-1)^{|j|_1|i|_2+|h|_2|i|_2+|j|_1|h|_2}\left\{\big(F_{1;h,j}(v)-F_{1;h,j}(u)\big)\big(F_{1;i,k}(u)-F_{1;i,k}(v)\big)\right\}_{d}.
	\end{aligned}
\end{equation}

We will prove \eqref{f1-0} by induction on $d$.
Take $\big\lbrace\:\big\rbrace_0$ on \eqref{f1-3'} so that 
\[
(u-v)^2\big\lbrace [F_{1;i,j}(u),F_{1;h,k}(v)]\big\rbrace_2=0.
\]
Note that the right-hand side of \eqref{f1-3} is zero when $u=v$, hence we may divide both sides by $(u-v)$ and therefore $(u-v)\big\lbrace [F_{1;i,j}(u),F_{1;h,k}(v)]\big\rbrace_2=0$, 
it proves \eqref{f1-0} is true for $d=1$.
Assuming that \eqref{f1-0} holds for some $d>1$,
that is,
\begin{equation}\label{f1-4}
	\begin{aligned}	
		&\left\{\big(F_{1;i,j}(u)-F_{1;i,j}(v)\big)\big(F_{1;h,k}(v)-F_{1;h,k}(u)\big)\right\}_{d}\\
		=&(-1)^{|j|_1|h|_2+|h|_2|i|_2+|j|_1|i|_2}(u-v)\left\{F_{1;h,j}(u)F_{1;i,k}(v)\right\}_{d+1}\\
		&-(-1)^{|j|_1|h|_2+|j|_1|k|_1+|k|_1|h|_2}(u-v)\left\{F_{1;i,k}(v)F_{1;h,j}(u)\right\}_{d+1}.
	\end{aligned}
\end{equation}

Take $\big\lbrace\:\big\rbrace_d$ on \eqref{f1-3},
and together with \eqref{f1-4},
we have
\begin{equation}\label{f1-5}
	\begin{aligned}
		&(u-v)^2\left\{[F_{1;i,j}(u),F_{1;k,l}(v)]\right\}_{d+2}=\left\{(u-v)^2[F_{1;i,j}(u),F_{1;k,l}(v)]\right\}_{d}\\
=&(-1)^{|j|_1|h|_2+|j|_1|i|_2+|h|_2|i|_2}(u{-}v)\left\{F_{1;h,j}(v)\left(F_{1;i,k}(u){-}F_{1;i,k}(v)\right)\right\}_{d+1}\\
&+(-1)^{|k|_1|j|_1+|k|_1|h|_2+|j|_1|h|_2+1}(u{-}v)\left\{F_{1;i,k}(u)F_{1;h,j}(u)\right\}_{d+1}\\
&+(-1)^{|j|_1|h|_2+|h|_2|i|_2+|j|_1|i|_2}(u-v)\left\{F_{1;h,j}(u)F_{1;i,k}(v)\right\}_{d+1}\\
	\end{aligned}
\end{equation}

Moreover,
the assumption gives
\begin{multline*}
(-1)^{|i|_2|j|_1+|h|_2|i|_2+|h|_2|j|_1}\left\{[F_{1;i,j}(u), F_{1;h,k}(v)]\right\}_{d+1}\\
	=\left\{\frac{\big(F_{1;h,j}(v)-F_{1;h,j}(u)\big)}{(u-v)}\big(F_{1;i,k}(u)-F_{1;i,k}(v)\big)\right\}_{d}.
\end{multline*}
Note that its right-hand side is zero when we set $u=v$, which implies that the element $\left\{[F_{1;i,j}(u), F_{1;h,k}(u)]\right\}_{d+1}$ is $0$,
then it gives
\begin{equation}\label{f1-6}
\begin{aligned}
\left\{F_{1;i,j}(u)F_{1;h,k}(u)\right\}_{d+1}=(-1)^{(|j|_1+|i|_2)(|k|_1+|h|_2)}\left\{F_{1;h,k}(u)F_{1;i,j}(u)\right\}_{d+1}.
\end{aligned}
\end{equation}		
Substituting \eqref{f1-6} into the last term in \eqref{f1-5},
we derive that
\begin{multline*}
(u-v)^2\left\{[F_{1;i,j}(u),F_{1;k,l}(v)]\right\}_{d+2}
=(-1)^{|j|_1|h|_2+|j|_1|i|_2+|h|_2|i|_2}(u{-}v)\\
\left\{\left(F_{1;h,j}(v)-F_{1;h,j}(u)\right)\left(F_{1;i,k}(u)
		-F_{1;i,k}(v)\right)\right\}_{d+1}.
\end{multline*}
So the relation \eqref{f1-0} is established.
\end{proof}

\section{Special Cases: $n=3$ and the super Serre relations}
In this section, we will consider the generators $D$'s, $E$'s and $F$'s in different super Yangians at the same time but using the same notation. It should be clear from the context which super Yangian we are dealing with.

We will describe a parabolic presentation of the Yangian $Y_{\mu}$ with $\mu=(\mu_1,\mu_2,\mu_3)$  using also the Gauss decomposition, 
and then use this to give the more general result in the next section. 
Similar to the proof of Proposition \ref{n=2}, 
we compute the matrix products
 \eqref{Tp=FDE} and \eqref{gaussdec} with respect to the composition $\mu=(\mu_1,\mu_2,\mu_3)$ as follows. 
$$^{\mu}T(u){=}\left(\begin{array}{ccc}
	D_1(u)&D_1(u)E_1(u)&D_1(u)E_{1,3}(u)\\
	F_1(u)D_1(u)&F_1(u)D_1(u)E_1(u)+D_2(u)&t_{23}\\
	F_{3,1}(u)D_1(u)&F_{3,1}(u)D_1(u)E_1(u){+}F_2(u)D_2(u)&t_{33}
\end{array}\right)$$
where $$t_{23}=F_1(u)D_1(u)E_{1,3}(u)+D_2(u)E_2(u),$$ $$t_{33}=F_{3,1}(u)D_1(u)E_{1,3}(u){+}F_2(u)D_2(u)E_2(u){+}D_3(u).$$
$$
^{\mu}T(u)^{-1}{=}
	\left(\begin{array}{ccc}
		t_{11}'(u)&{-}E_1(u)D_2(u)^{-1}{-}E'_{1,3}(u)D_3(u)^{-1}F_2(u)&E'_{1,3}(u)D_3(u)^{-1}\\
		t_{21}'(u)&D_2(u)^{-1}+E_2(u)D_3(u)^{-1}F_2(u)&-E_2(u)D_3(u)^{-1}\\
		D_3(u)^{-1}F'_{3,1}(u)&-D_3(u)^{-1}F_{2}(u)&D_3(u)^{-1}
	\end{array}\right),
$$
where
$$t_{11}'(u)=D_1(u)^{-1}{+}E_1(u)D_2(u)^{-1}F_1(u)
{+}E'_{1,3}(u)D_3(u)^{-1}F'_{3,1}(u),$$
$$t_{21}'(u)={-}D_2(u)^{-1}F_1(u){-}E_2(u)D_3(u)^{-1}F'_{3,1}(u),$$
and the notations
$$E'_{1,3}(u)=E_1(u)E_2(u)-E_{1,3}(u),\quad F'_{3,1}(u)=F_2(u)F_1(u)-F_{3,1}(u).$$
These expressions allow us to derive the following identities.
\begin{eqnarray}
\label{T11}&t_{i,j}(u)=D_{1;i,j}(u),&\mbox{for}\,1{\leq}i,j{\leq}\mu_1,\\
\label{T12}&t_{i,\mu_1+j}(u)=\sum_{p=1}^{\mu_1}D_{1;i,p}E_{1;p,j}(u),&\mbox{for}\,1{\leq}i{\leq}\mu_1,\,1{\leq}j{\leq}\mu_2,\\
\label{T21}&t_{\mu_1+i,j}(u)=\sum_{p=1}^{\mu_1}F_{1;i,p}(u)D_{1;p,j}(u),&\mbox{for}\,1{\leq}i{\leq}\mu_2,\,1{\leq}j{\leq}\mu_1,\\
\label{T31}&t_{\mu_1+\mu_2+i,j}(u)=\sum_{p=1}^{\mu_1}F_{3,1;i,p}(u)D_{1;p,j}(u),&\mbox{for}\,1{\leq}i{\leq}\mu_3,\,1{\leq}j{\leq}\mu_1,\\
\label{TP33}&t^{\prime}_{\mu_1+\mu_2+i,\mu_1+\mu_2+j}(u)=D'_{3;i,j}(u),&\mbox{for}\,1{\leq}i,j{\leq}\mu_3,\\
\label{TP32}&t^{\prime}_{\mu_1+\mu_2+i,\mu_1+j}(u)=-\sum_{q=1}^{\mu_3}D'_{3;i,q}(u)F_{2;q,j}(u),&\mbox{for}\,1{\leq}i{\leq}\mu_3,\,1{\leq}j{\leq}\mu_2,\\
\label{TP31}&t^{\prime}_{\mu_1+\mu_2+i,j}(u)=\sum_{q=1}^{\mu_3}D'_{3;i,q}(u)F'_{3,1;q,j}(u),&\mbox{for}\,1{\leq}i{\leq}\mu_3,\,1{\leq}j{\leq}\mu_1,\\
	\label{TP23}&t^{\prime}_{\mu_1+i,\mu_1+\mu_2+j}(u)=-\sum_{q=1}^{\mu_3}E_{2;i,p'}(u)D^{\prime}_{3;q,j}(u),&\mbox{for}\,1{\leq}i{\leq}\mu_2,\,1{\leq}j{\leq}\mu_3
\end{eqnarray}

In the case $\mu=(\mu_1,\mu_2,\mu_3)$,
the relations for $D$'s and $F$'s are left to be considered as follows:

\begin{proposition}\label{3DF}
	The following identities hold in $Y_{(\mu_1,\mu_2,\mu_3)}((u^{-1},v^{-1}))$:
\begin{eqnarray}
\label{D1F2} [D_{1;i,j}(u),F_{2;h,k}(v)]&{=}&0\quad \mbox{for}~ 1{\leq}i,j{\leq}\mu_1, \,1{\leq}h{\leq}\mu_3,\,1{\leq}k{\leq}\mu_2,
\\
\label{F1D3}	[F_{1;i,j}(u),D^{\prime}_{3;h,k}(v)]&{=}&0\quad \mbox{for}~ 1{\leq}i{\leq}\mu_2,\,1{\leq}j{\leq}\mu_1,\,1{\leq}h,k{\leq}\mu_3,
\end{eqnarray}
\begin{multline}\label{D3F2}
[F_{2;i,j}(u), D^{\prime}_{3;h,k}(v)]
=\dfrac{(-1)^{|h|_3|i|_3+|h|_3|j|_2+|j|_2|k|_3}}{(u-v)}\del_{ik}
D^{\prime}_{3;h,p}(v)\big(F_{2;p,j}(v){-}F_{2;p,j}(u)\big).
\end{multline}
\begin{multline}\label{D3F31}
[F_{3,1;i,j}(u), D^{\prime}_{3;h,k}(v)]
=-(-1)^{|j|_1|i|_3+|j|_1|r|_2+|h|_3|i|_3+|m|_3|r|_2+|h|_3|j|_1+|m|_3|j|_1}\\\del_{ik}D^{\prime}_{3;h,m}(v)[F_{1;r,j}(u),F_{2;m,r}(v)]
\end{multline}
where the indices $p$ is summed over $1,\cdots,\mu_3$.
\end{proposition}
\begin{proof}
	We obtain $[t_{i,j}(u),t^{\prime}_{\mu_1+\mu_2+h,\mu_1+k}(v)]{=}0$ and $[t_{\mu_1+i,j}(u),t^{\prime}_{\mu_1+\mu_2+h,\mu_1+\mu_2+k}(v)]{=}0$ by \eqref{commurelation}. Substituting by \eqref{T11} and \eqref{TP32}, \eqref{T21} and \eqref{TP33}, respectively,
	then using the fact that $D_{1;i,j}(u)$, $D^{\prime}_{3;h,k}(v)$ supercommute,
	it gives \eqref{D1F2} and \eqref{F1D3}.
	
To show \eqref{D3F2},	
the identity \eqref{f1d2} in $Y_{(\mu_2,\mu_3)}((u^{-1},v^{-1}))$ reads as	
	\begin{multline*}
		(u-v)[F_{1;i,j}(u), D^{\prime}_{2;h,k}(v)]
		=(-1)^{|h|_2|i|_2+|h|_2|j|_1+|j|_1|k|_2}\del_{ik}\times\\
		\sum_{m=1}^{\mu_2}D^{\prime}_{2;h,m}(v)\big(F_{1;m,j}(v)-F_{1;m,j}(u)\big).
	\end{multline*}
	Applying the map $\psi_{\mu_1}$ to this identity and using \eqref{pqele}, it gives \eqref{D3F2}.
	
To show \eqref{D3F31},	 by \eqref{terpow2},
	we have 
	$
	F_{3,1;i,j}(u)=(-1)^{|r|_{2}}[F_{2;i,r}^{(1)}, F_{1;r,j}(u)]
	$ for any $1{\leq}r{\leq}\mu_2$,
	then by the super Jacobi identity, \eqref{F1D3} and \eqref{u0-coeff},
	we derive it.
	\end{proof}

\begin{proposition}\label{3FF}
	The following identities hold in $Y_{(\mu_1,\mu_2,\mu_3)}((u^{-1},v^{-1}))$:
	\begin{eqnarray}
		\label{F1E2}[F_{1;i,j}(u),E_{2;h,k}(v)]&=&0\quad \mbox{for}~1{\leq}i,h{\leq}\mu_2,\, 1{\leq}j{\leq}\mu_1,\, 1{\leq}k{\leq}\mu_3,
	\end{eqnarray}
	\begin{multline}\label{F1F2}
		(u-v)[F_{1;i,j}(u),F_{2;h,k}(v)]=
		(-1)^{|i|_2|j|_1+|i|_2|h|_3+|j|_1|h|_3}\delta_{ik}\times\\
		\big\lbrace \sum_{q=1}^{\mu_2}F_{2;h,q}(v)\big(F_{1;q,j}(v)-F_{1;q,j}(u)\big)-F_{3,1;h,j}(v)+F_{3,1;h,j}(u)\big\rbrace,
	\end{multline}
	\begin{multline}\label{F2F3}
		[F_{3,1;i,j}(u),F_{2;h,k}(v)]=
		{-}(-1)^{|i|_3|j|_1{+}|i|_3|h|_3{+}|j|_1|h|_3{+}|g|_2}
		[F_{2;h,g}(v),F_{1;g,j}(u)]F_{2;i,k}(v),
	\end{multline}
	\begin{multline}\label{F1F31}
		[F_{1;i,j}(u) , F'_{3,1;h,k}(v)]=
		(-1)^{(|h|_3+|j|_1)(|k|_1+|g|_2)} F_{1;i,k}(u)[F_{1;g,j}(u),F_{2;h,g}(v)].
	\end{multline}
	Here \eqref{F1F2} holds for all $1{\leq}i,k{\leq}\mu_2$, $1{\leq}j{\leq}\mu_1$, $1{\leq}h{\leq}\mu_3$, \eqref{F2F3} holds for all $1{\leq}i,h{\leq}\mu_3$, $1{\leq}j{\leq}\mu_1$, $1{\leq}j,k{\leq}\mu_2$, and \eqref{F1F31} holds for all $1{\leq}i,g{\leq}\mu_2$, $1{\leq}j,k{\leq}\mu_1$, $1{\leq}h{\leq}\mu_3$.
\end{proposition}

\begin{proof}
By \eqref{commurelation}, we have $[t_{\mu_1+i,j}(u),t^{\prime}_{\mu_1+h,\mu_1+\mu_2+k}(v)]=0$. Substituting by \eqref{T21} and \eqref{TP23}, we have
\[
[F_{1;i,p}(u)D_{1;p,j}(u),-E_{2;h,q}(v)D^{\prime}_{3;q,k}(v)]=0.
\]
Computing the bracket, we obtain
\begin{multline}\label{F1E2-1}
F_{1;i,p}(u)D_{1;p,j}(u)E_{2;h,q}(v)D^{\prime}_{3;q,k}(v)-\\
	(-1)^{(|i|_2+|j|_1)(|h|_2+|k|_3)}E_{2;h,q}(v)D^{\prime}_{3;q,k}(v)F_{1;i,p}(u)D_{1;p,j}(u)=0,
\end{multline}
where $p$ and $q$ are summed over $1,\ldots,\mu_1$ and $1,\ldots,\mu_3$, respectively.

Substituting \eqref{F1D3} and  \eqref{D1E2} into \eqref{F1E2-1} and using the fact that $D_{1;i,j}(u)$, $D^{\prime}_{3;h,k}(v)$ supercommute again,
we derive that
\begin{multline}\label{F1E2-2}
(-1)^{|p|_1|h|_2}F_{1;i,p}(u)E_{2;h,q}(v)D^{\prime}_{3;q,k}(v)D_{1;p,j}(u)-\\
	(-1)^{|i|_2|h|_2+|i|_2|q|_3+|p|_1|q|_3}E_{2;h,q}(v)F_{1;i,p}(u)D^{\prime}_{3;q,k}(v)D_{1;p,j}(u)=0.
\end{multline}
The sign factors are free from the indices $j$ and $k$.
Multiplying $D'_1(u)D_3(v)$ from the
right, we obtain \eqref{F1E2}.

By \eqref{commurelation} again, for $1{\leq}i,k\leq\mu_2$, $1{\leq}j\leq\mu_1$ and $1{\leq}h\leq\mu_3$,
we have
\[
(u-v)[t_{\mu_1{+}i,j}(u),t^{\prime}_{\mu_1{+}\mu_2{+}h,\mu_1{+}k}(v)]=-(-1)^{|i|_2|j|_1{+}|i|_2|h|_3{+}|j|_1|h|_3}\delta_{ik}\sum_{s=1}^{M{+}N}
t^{\prime}_{\mu_1{+}\mu_2{+}h,s}(v)t_{s,j}(u).
\]
Substituting by \eqref{T11}--\eqref{TP31}, we have
\begin{multline}\label{F1F2-1}
	(u-v)D'_{3;h,q}(v)F_{1;i,p}(u),F_{2;q,k}(v)]D_{1;p,j}(u)=(-1)^{|i|_2|q|_3+|p|_1|q|_3+|p|_1|k|_2}\delta_{ik}\\
	D'_{3;h,q}(v)\left[F'_{3,1;q,p}(v)
	-F_{2;q,r}(v)F_{1;r,p}(u)
	+F_{3,1;q,p}(u)\right]D_{1;p,j}(u),
\end{multline}
The sign factors are free from the indices $h$ and $j$. Canceling $D'_{1}(u)$ from the right and $D_3(v)$ from the left on both sides of \eqref{F1F2-1} and dividing both sides by $u-v$, we have deduced 
\begin{multline}\label{F1F2-2}
	(u-v)[F_{1;i,p}(u),F_{2;q,k}(v)]=(-1)^{|i|_2|q|_3+|p|_1|q|_3+|p|_1|k|_2}\delta_{ik}\\
	\left[F'_{3,1;q,p}(v)
	-F_{2;q,r}(v)F_{1;r,p}(u)
	+F_{3,1;q,p}(u)\right].
\end{multline}

Substitue $F'_{3,1;q,p}(v)=\sum_{r=1}^{\mu_2}F_{2;q,r}(v)F_{1;r,p}(v)-F_{3,1;q,p}(v)$ obtained by $F'_{3,1}(u)=F_2(u)F_1(u)-F_{3,1}(u)$
into the identity \eqref{F1F2-2},
it gives \eqref{F1F2}.
 
To show \eqref{F2F3},
we take the coefficient of $u^0$ in the identity \eqref{D3F2}, and then obtain
\begin{equation}\label{u0-coeff}
[F^{(1)}_{2;i,j}(u), D^{\prime}_{3;h,k}(v)]=(-1)^{|h|_3|i|_3+|h|_3|j|_2+|j|_2|k|_3}\del_{ik}D^{\prime}_{3;h,m}(v)F_{2;m,j}(v).
\end{equation}

 By \eqref{commurelation}, \eqref{T31} and \eqref{TP32}, we have
\[
[t_{\mu_1+\mu_2+i,j}(u),t^{\prime}_{\mu_1+\mu_2+h,\mu_1+k}(v)]=[F_{3,1;i,p}(u)D_{1;p,j}(u),-D'_{3;h,q}(v)F_{2;q,k}(v)]=0.
\] 
Then associated to \eqref{D1F2} and the fact that $D_{1;i,j}(u)$, $D^{\prime}_{3;h,k}(v)$ supercommute,
multiplying $D'_1(u)$ from the right,
we have 
$[F_{3,1;i,j}(u),D'_{3;h,q}(v)F_{2;q,k}(v)]=0$,
which may be written as \begin{multline*}
	[F_{3,1;i,j}(u),D'_{3;h,q}(v)]F_{2;q,k}(v)\\
	+(-1)^{(|i|_3+|j|_1)(|h|_3+|q|_3)}D'_{3;h,q}(v)[F_{3,1;i,j}(u),F_{2;q,k}(v)]=0.
\end{multline*}
Then associated to \eqref{D3F31},
we obtain
\begin{multline*}
	(-1)^{|i|_3|q|_3+|j|_1|q|_3}D'_{3;h,q}(v)[F_{3,1;i,j}(u),F_{2;q,k}(v)]=\\	(-1)^{|j|_1|i|_3+|j|_1|r|_2+|m|_3|r|_2+|m|_3|j|_1}D^{\prime}_{3;h,m}(v)[F_{1;r,j}(u),F_{2;m,r}(v)]F_{2;i,k}(v),
\end{multline*}
and here the sign is free to $h$,
then \eqref{F2F3} is given.

At last,
to obtain \eqref{F1F31},
we take firstly the coefficient of $u^{0}$ in \eqref{F1F2}, and get
\begin{equation}\label{F1F31-1}
[F^{(1)}_{1;i,k},F_{2;h,j}(v)]=(-1)^{|i|_2|k|_1{+}|i|_2|h|_3{+}|k|_1|h|_3}\delta_{ij}F'_{3,1;h,k}(v).
\end{equation}

Taking the coefficient of $v^{0}$ in \eqref{d1f1}, we have
\begin{equation}\label{F1F31-2}
	[D_{1;i,j}(u),F^{(1)}_{1;a,k}]=-(-1)^{|i|_1|j|_1{+}|a|_2|i|_1{+}|a|_2|j|_1}\delta_{ik}F_{1;a,p}(u)D_{1;p,j}(u).
\end{equation}

Together with the super Jacobi identity and the relation \eqref{D1F2},
\eqref{F1F31-1} and \eqref{F1F31-2} imply that
\begin{multline}\label{F1F31-3}
	D_{1;i,j}(u)F'_{3,1;h,k}(v)=(-1)^{(|i|_1+|j|_1)(|h|_3+|k|_1)}F'_{3,1;h,k}(v)D_{1;i,j}(u)\\
-(-1)^{|a|_2|h|_3{+}|i|_1|h|_3{+}|i|_1|j|_1{+}|p|_1|a|_2{+}|p|_1|h|_3{+}|j|_1|h|_3}\delta_{ik}[F_{1;a,p}(u),F_{2;h,a}(v)]D_{1;p,j}(u).
\end{multline}

By \eqref{commurelation}, \eqref{T21} and \eqref{T31},
we have
\begin{multline*}
[t_{\mu_1{+}i,j}(u),t'_{\mu_1{+}\mu_2{+}h,k}(v)]
=[F_{1;i,p}(u)D_{1;p,j}(u),D'_{3;h,q}(v)F'_{3,1;q,k}(v)]=0,
\end{multline*}
then $D'_3$ supercommutes with $D_1$ and $F_1$,
then multiplying $D_3(u)$ from the left,
and computing the bracket,
we obtain
\begin{multline}\label{F1F31-4}
F_{1;i,g}(u)D_{1;g,j}(u)F'_{3,1;t,k}(v)
	-(-1)^{(|i|_2+|j|_1)(|t|_3+|k|_1)}F'_{3,1;t,k}(v)F_{1;i,g}(u)D_{1;g,j}(u)=0,
\end{multline}
Then substituting \eqref{F1F31-3} into the first term of 
\eqref{F1F31-4}, 
we have
\begin{multline*}
(-1)^{|g|_1(|t|_3+|k|_1)}\left[F_{1;i,g}(u),F'_{3,1;t,k}(v)\right]D_{1;g,j}(u)\\
\\
=(-1)^{|a|_2|t|_3{+}|k|_1|t|_3{+}|p|_1|a|_2{+}|p|_1|t|_3}F_{1;i,k}(u)[F_{1;a,p}(u),F_{2;t,a}(v)]D_{1;p,j}(u)
\end{multline*}
The sign is free for $j$,
then multiplying $D'_1(u)$ from the right,
it establishes \eqref{F1F31}.
\end{proof}

On the one hand,
the following Proposition \ref{3DE} can be proved in a similar way with Proposition \ref{3DF} and \ref{3FF}.
Such as, we can use \eqref{commurelation} to obtain $$[t_{i,j}(u),t^{\prime}_{\mu_1+h,\mu_1+\mu_2+k}(v)]{=}0\quad\mbox{and}\quad [t_{i,\mu_1+j}(u),t^{\prime}_{\mu_1+h,\mu_1+\mu_2+k}(v)]{=}0,$$
then substituting by \eqref{T11} and \eqref{TP23}, 
\eqref{T12} and \eqref{TP33}, respectively, 
it gives the relation \eqref{D1E2} and \eqref{D3E1} with the fact that $D_{1;i,j}(u)$, $D^{\prime}_{3;h,k}(v)$ supercommute.
On the other hand,
applying the automorphism $\zeta_{N|M}$ to the above corresponding identities in $Y_{\overleftarrow{\mu}}=Y_{(\mu_3,\mu_2,\mu_1)}=Y_{N|M}$,
Proposition \ref{3DE} also can be deduced.

\begin{proposition}\label{3DE}
	The following identities hold in $Y_{(\mu_1,\mu_2,\mu_3)}((u^{-1},v^{-1}))$:
	\begin{eqnarray}
		\label{D1E2} [D_{1;i,j}(u),E_{2;h,k}(v)]{=}&0\quad \mbox{for}~ 1{\leq}i,j{\leq}\mu_1, \,1{\leq}h{\leq}\mu_2,\,1{\leq}k{\leq}\mu_3,\\
		\label{D3E1} [E_{1;i,j}(u),D'_{3;h,k}(v)]{=}&0\quad \mbox{for}~  1{\leq}i{\leq}\mu_1,\,1{\leq}j{\leq}\mu_2,\,1{\leq}h,k{\leq}\mu_3,\\
		\label{E1F2} [E_{1;i,j}(u), F_{2;h,k}(v)]{=}&0\quad \mbox{for}~ 1{\leq}i{\leq}\mu_1, \,1{\leq}j,k{\leq}\mu_2,\,1{\leq}h{\leq}\mu_3,\\
		\label{D3E2} [E_{2;i,j}(u),D^{\prime}_{3;h,k}(v)]{=}&\dfrac{(-1)^{|j|_3|h|_3}}{(u-v)}\delta_{jh}\big(E_{2;i,g}(u)-E_{2;i,g}(v)\big)D^{\prime}_{3;g,k}(v),\\
		\label{D3E13}[E_{1,3;i,j}(u),D^{\prime}_{3;h,k}(v)]{=}&(-1)^{1{+}|g|_2{+}|j|_3|h|_3}\delta_{hj}\big[E_{1;i,q}(u),E_{2;q,g}(v)\big]D^{\prime}_{3;g,k}(v)\\
		\label{D1E13}[D_{1;i,j}(u),E'_{1,3;h,k}(v)]{=}&(-1)^{1{+}|g|_2{+}|j|_1|h|_1}\delta_{hj}D_{1;i,p}(u)[E_{1;p,q}(v),E_{2;q,k}(v)],	\end{eqnarray}
	\begin{multline}\label{E1E2}
		(u-v)[E_{1;i,j}(u), E_{2;h,k}(v)] = \\
		(-1)^{|j|_2|h|_2}\delta_{hj}\big\lbrace\big(E_{1;i,q}(u)-E_{1;i,q}(v)\big)E_{2;q,k}(v)
		+ E_{1,3;i,k}(v) - E_{1,3;i,k}(u)\big\rbrace,
	\end{multline}
	\begin{multline}\label{61c}
		[E_{1,3;i,j}(u), E_{2;h,k}(v)] =\\
		(-1)^{|i|_1|j|_3+|i|_1|h|_2+|h|_2|j|_3+|g|_2}
		E_{2;h,j}(v) [E_{1;i,g}(u), E_{2;g,k}(v)],
	\end{multline}
	\begin{multline}\label{61d}
		\sum_{q=1}^{\mu_2}[E_{1;i,j}(u), E'_{1,3;h,k}(v)] =\\
		(-1)^{|h|_1|j|_2+|j|_2|k|_3+|h|_1|k|_3+|g|_2}
		[E_{1;i,g}(u), E_{2;g,k}(v)] E_{1;h,j}(u).
	\end{multline}
\end{proposition}
\section{The general case}\label{sec:general}
Recall that we are aim to obtain the defining relations of $Y_{\mu}(\so)$ in terms of those above parabolic generators associated to general $\mu=(\mu_1,\cdots,\mu_n)$ and $\so$.
Except for the identities in Proposition \ref{dd0},
many of the relations among the functions of $D,E$ and $F$ in $Y_{\mu}(\so)$ are given by the ones in $Y_{\mu_1,\mu_2}[[u^{-1},v^{-1}]]$ and $Y_{\mu_1,\mu_2,\mu_3}[[u^{-1},v^{-1}]]$ together with the shift maps and swap maps.
Roughly speaking, together with Proposition \ref{dd0}, \eqref{d1e1} and \eqref{d2e1} give \eqref{gde},
\eqref{d1f1} and \eqref{d2f1} give \eqref{gdf},
\eqref{E1E2} and \eqref{F1F2} give \eqref{gee-1} and \eqref{gff-1}, respectively.
\eqref{gef} follows from \eqref{e1f1}, \eqref{F1E2} and \eqref{E1F2}.
Applying to \eqref{e1e1} and \eqref{f1f1} produce \eqref{gee} and \eqref{gff}, respectively.
The relations \eqref{gee-1} and \eqref{gff-1} are followed from \eqref{E1E2} and \eqref{F1F2},
respectively.
Then we just go through \eqref{serre-F}, \eqref{super-F} and \eqref{superserre-F} in the following lemma,
which is a parabolic version of \cite[Lemma 3.7]{CH}, in series forms in $Y_{\mu}(\so)$.
Note that \eqref{superserre-E} and \eqref{superserre-F} are called the super Serre relations for $n\geq 4$.
\begin{lemma}\label{identity-mu}
The following relations
	 hold in Yangian $Y_{\mu}(\so)$:
\begin{equation}\label{EEFF}
[E_{a;i,j}(v),E_{a;h,k}(v)]=0,\quad [F_{a;i,j}(v),F_{a;h,k}(v)]=0,	
\end{equation} 
\begin{multline}\label{gde}
(u-v)[D_{a;i,j}(u), E_{b;h,k}(v)]
=(-1)^{|h|_a|j|_a}\delta_{a,b}\delta_{hj}\sum_{p=1}^{\mu_a}D_{a;i,p}(u)\big(E_{b;p,k}(v)-E_{b;p,k}(u)\big)\\
	+\delta_{a,b+1}(-1)^{|h|_b|k|_{a}+|h|_b|j|_a+|j|_a|k|_{a}}
	 D_{a;i,k}(u)\big(E_{b;h,j}(u)-E_{b;h,j}(v)\big),
\end{multline}
\begin{multline}\label{gdf}
(u-v)[D_{a;i,j}(u), F_{b;h,k}(v)]
{=}(-1)^{|i|_a|j|_a+|h|_{a+1}|i|_a+|h|_{a+1}|j|_a}\delta_{a,b}\delta_{ik}\sum_{p=1}^{\mu_1}\big(F_{b;h,p}(u)-F_{b;h,p}(v)\big)\\D_{a;p,j}(u)
+(-1)^{|h|_a|k|_{a-1}+|h|_a|j|_a+|j|_a|k|_{a-1}}\delta_{a,b+1}\big(F_{b;i,k}(v){-}F_{b;i,k}(u)\big)D_{a;h,j}(u), 
\end{multline}
\begin{multline}\label{gef}
(u-v)[E_{a;i,j}(u), F_{b;h,k}(v)]
=\delta_{a,b}\big\lbrace(-1)^{|h|_{a+1}|i|_a+|i|_a|j|_{a+1}+|h|_{a+1}|j|_a}D_{a+1;h,j}(u) D^{\prime}_{a;i,k}(u)\\
-(-1)^{|h|_{a+1}|k|_a+|j|_{a+1}|k|_a+|h|_{a+1}|j|_{a+1}}
D^{\prime}_{a;i,k}(v) D_{a+1;h,j}(v)\big\rbrace,
\end{multline}  
\begin{multline}\label{gee}       
(u-v)[E_{a;i,j}(u), E_{a;h,k}(v)]=(-1)^{|h|_a|j|_{a+1}+|j|_{a+1}|k|_{a+1}+|h|_a|k|_{a+1}}\\
\big(E_{a;i,k}(u)-E_{a;i,k}(v)\big)\big(E_{a;h,j}(u)-E_{a;h,j}(v)\big),\end{multline}
\begin{multline}\label{gff}	
(u-v)[F_{a;i,j}(u), F_{a;h,k}(v)]=-(-1)^{|i|_{a+1}|j|_a+|h|_{a+1}|i|_{a+1}+|h|_{a+1}|j|_a}\\
\big(F_{a;h,j}(u)-F_{a;h,j}(v)\big)\big(F_{a;i,k}(u)-F_{a;i,k}(v)\big),
\end{multline}
\begin{multline}\label{gee-1}
(u-v)[E_{a;i,j}(u), E_{a+1;h,k}(v)]=(-1)^{|j|_{a+1}|h|_{a+1}}\delta_{h,j}\\
\big\lbrace\big(E_{a;i,q}(u)-E_{a;i,q}(v)\big)E_{a+1;q,k}(v)+E_{a,a+2;i,k}(v)-E_{a,a+2;i,k}(u)\big\rbrace,
\end{multline}
\begin{multline}\label{gff-1}
(u-v)[F_{a;i,j}(u),F_{a+1;h,k}(v)]
=(-1)^{|i|_{a+1}|j|_a+|i|_{a+1}|h|_{a+2}+|j|_a|h|_{a+2}}
\delta_{i,k}\\
\big\lbrace\sum_{q=1}^{\mu_2}F_{a+1;h,q}(v)\big(F_{a;q,j}(v)-F_{a;q,j}(u)\big)-F_{a+2,a;h,j}(v)+F_{a+2,a;h,j}(u)\big\rbrace,
\end{multline}
\begin{equation}\label{serre-E}
	\big[[E_{a;i,j}(u),E_{b;h,k}(v)],E_{b;f,g}(v)\big]=0,\end{equation}
\begin{equation}\label{super-E}
	\big[[E_{a;i,j}(u),E_{b;h,k}(v)],E_{b;f,g}(w)\big]{+}
	\big[[E_{a;i,j}(u),E_{b;h,k}(w)],E_{b;f,g}(v)\big]{=}0,\end{equation}
\begin{equation}\label{serre-F}
	\big[[F_{a;i,j}(u),F_{b;h,k}(v)],F_{b;f,g}(v)\big]=0,\end{equation}
\begin{equation}\label{super-F}
	\big[[F_{a;i,j}(u),F_{b;h,k}(v)],F_{b;f,g}(w)\big]+\big[[F_{a;i,j}(u),F_{b;h,k}(w)],F_{b;f,g}(v)\big]=0
\end{equation}
where $|a-b|=1$.
\begin{equation}\label{superserre-E}
\big[\,[E_{a;i,f_1}^{(r)},E_{a+1;f_2,j}^{(1)}],[E_{a+1;h,g_1}^{(1)},E_{a+2;g_2,k}^{(s)}]\,\big]=0,
\end{equation}
\begin{equation}\label{superserre-F}
\big[\,[F_{a;f_1,i}^{(r)},F_{a+1;j,f_2}^{(1)}],[F_{a+1;g_1,h}^{(1)},F_{a+2;k,g_2}^{(s)}]\,\big]=0
\end{equation}
where $n\geq 4$.
\end{lemma}
\begin{proof}
We just need to consider in detail \eqref{serre-F} and \eqref{super-F} in the case $a=1,b=2$ in $Y_{(\mu_1,\mu_2,\mu_3)}[[u^{-1},v^{-1},w^{-1}]]$
and \eqref{superserre-F} in the case $n=4$,
while \eqref{serre-E}, \eqref{super-E} and \eqref{superserre-E} are similar.

To show \eqref{serre-F},
we start from the following relation
$$\big[[F_{1;i,j}(u),F_{2;h,k}(v)],F_{2;f,g}(v)\big]{=}(-1)^{(|j|_1{+}|i|_2)(|k|_2{+}|h|_3)}\big[F_{2;h,k}(v),[F_{1;i,j}(u),F_{2;f,g}(v)]\big]$$  
obtained by the super Jacobi identity  and the relation \eqref{EEFF},
then together with \eqref{F1F2},
it suffices to prove the case that $i=k=g$.
In this case,
we have 
\begin{multline*}
(u-v)\big[[F_{1;i,j}(u),F_{2;h,i}(v)],F_{2;f,i}(v)\big]
=(-1)^{|i|_2|h|_3{+}|h|_3|f|_3{+}|i|_2|f|_3}\\
\left(\big[[F_{1;i,j}(u),F_{2;f,i}(v)],F_{2;h,i}(v)\big]
-\big[[F_{1;i,j}(v),F_{2;f,i}(v)],F_{2;h,i}(v)\big]\right)
\end{multline*}
Thus we have 
\begin{multline}\label{serre-F-1}
(u-v-(-1)^{|i|_2|h|_3{+}|h|_3|f|_3{+}|i|_2|f|_3})\big[[F_{1;i,j}(u),F_{2;h,i}(v)],F_{2;f,i}(v)\big]\\
=-(-1)^{|i|_2|h|_3{+}|h|_3|f|_3{+}|i|_2|f|_3}\big[[F_{1;i,j}(v),F_{2;f,i}(v)],F_{2;h,i}(v)\big].
\end{multline}
Note that the right-hand side of \eqref{serre-F-1} is independent of the choice of $u$.
Specifying $u=v+(-1)^{|i|_2|h|_3{+}|h|_3|f|_3{+}|i|_2|f|_3}$ in \eqref{serre-F-1},
we have
\[0=-(-1)^{|i|_2|h|_3{+}|h|_3|f|_3{+}|i|_2|f|_3}\big[[F_{1;i,j}(v),F_{2;f,i}(v)],F_{2;h,i}(v)\big],\]
and hence for any $u$,
\[(u-v-(-1)^{|i|_2|h|_3{+}|h|_3|f|_3{+}|i|_2|f|_3})\big[[F_{1;i,j}(u),F_{2;h,i}(v)],F_{2;f,i}(v)\big]=0.\]
Choose $u$ such that $u-v-(-1)^{|i|_2|h|_3{+}|h|_3|f|_3{+}|i|_2|f|_3}$ is invertible,
and then \eqref{serre-F} follows.

To show \eqref{super-F},
it suffices to prove that
\begin{equation}\label{super-F-0}(u-w)(v-w)(u-v)\big[[F_{1;i,j}(u),F_{2;h,k}(v)],F_{2;f,g}(w)\big]\end{equation} is symmetric in $v$ and $w$.
As in the proof of \eqref{serre-F},
we also may asuume $i=k$ ,
and then computing the brackets by \eqref{F1F2},
we have 
\begin{multline*}
	(u-w)(v-w)(u-v)\big[[F_{1;i,j}(u),F_{2;h,i}(v)],F_{2;f,g}(w)\big]
	=(-1)^{|i|_2|j|_1+|i|_2|h|_3+|j|_1|h|_3}\\
(u-w)(v-w)\big[\sum_{q=1}^{\mu_2}\left(F_{2;h,q}(v)F_{1;q,j}(v)-F_{2;h,q}(v)F_{1;q,j}(u)\right)
+F_{3,1;h,j}(u)-F_{3,1;h,j}(v),F_{2;f,g}(w)\big].\end{multline*}
Owing to 
\begin{multline*}
\sum_{q=1}^{\mu_2}\big[F_{2;h,q}(v)F_{1;q,j}(v),F_{2;f,g}(w)\big]=F_{2;h,g}(v)\big[F_{1;g,j}(v),F_{2;f,g}(w)\big]\\
+\sum_{q=1}^{\mu_2}(-1)^{(|j|_1+|q|_2)(|f|_3+|g|_2)}\big[F_{2;h,q}(v),F_{2;f,g}(w)\big]F_{1;q,j}(v).
\end{multline*}
Then together with \eqref{F2F3},
\eqref{super-F-0} equals to 
\begin{align*}
&(-1)^{|i|_2|j|_1+|i|_2|h|_3+|j|_1|h|_3}
(u-w)(v-w)\\
&\{(-1)^{(|j|_1+|q|_2)(|f|_3+|g|_2)}\big[F_{2;h,q}(v),F_{2;f,g}(w)\big]F_{1;q,j}(v)\\
&-(-1)^{(|j|_1+|q|_2)(|f|_3+|g|_2)}\big[F_{2;h,q}(v),F_{2;f,g}(w)\big]F_{1;q,j}(u)\\
&+F_{2;h,g}(v)\big[F_{1;g,j}(v),F_{2;f,g}(w)\big]-F_{2;h,g}(v)\big[F_{1;g,j}(u),F_{2;f,g}(w)\big]\\
&-(-1)^{|h|_3|j|_1+|f|_3|h|_3+|j|_1|f|_3+|a|_2}\big[F_{2;f,a}(w),F_{1;a,j}(u)\big]F_{2;h,g}(w)\\
&+(-1)^{|h|_3|j|_1+|f|_3|h|_3+|j|_1|f|_3+|a|_2}\big[F_{2;f,a}(w),F_{1;a,j}(v)\big]F_{2;h,g}(w)\}\\
=&(-1)^{|i|_2|j|_1+|i|_2|h|_3+|j|_1|h|_3}
(u-w)(v-w)\\
&\{-(-1)^{(|j|_1+|h|_3)(|f|_3+|g|_2)}\big[F_{2;f,g}(w),F_{2;h,q}(v)\big]F_{1;q,j}(v)\\
&+(-1)^{(|j|_1+|h|_3)(|f|_3+|g|_2)}\big[F_{2;h,q}(v),F_{2;f,g}(w)\big]F_{1;q,j}(u)\\
&+F_{2;h,g}(v)\big[F_{1;g,j}(v),F_{2;f,g}(w)\big]-F_{2;h,g}(v)\big[F_{1;g,j}(u),F_{2;f,g}(w)\big]\\
&+(-1)^{|g|_2|j|_1+|f|_3|j|_1+|g|_2|f|_3+|a|_2}(-1)^{(|j|_1+|a|_2)(|f|_3+|a|_2)}F_{2;h,g}(w)\big[F_{1;a,j}(u),F_{2;f,a}(w)\big]\\
&-(-1)^{|g|_2|j|_1+|f|_3|j|_1+|g|_2|f|_3+|a|_2}(-1)^{(|j|_1+|a|_2)(|f|_3+|a|_2)}F_{2;h,g}(w)\big[F_{1;a,j}(v),F_{2;f,a}(w)\big]\}
\end{align*}
where the second identity is derived from the following  
\begin{multline*}
	[F_{2;h,i}(v),F_{1;i,j}(u)]F_{2;f,g}(v)=(-1)^{(|g|_2+|f|_3)(|j|_1+|h|_3)}F_{2;f,g}(v)[F_{2;h,i}(v),F_{1;i,j}(u)]
\end{multline*}
obtained by \eqref{serre-F}.

Then we use \eqref{f1f1} and \eqref{F1F2} to compute these brackets,
and continue to obtain that \eqref{super-F-0} equals to 
\begin{align*}
&\epsilon\{(u-w)\big(F_{2;h,g}(v)-F_{2;h,g}(w)\big)\big(F_{2;f,q}(w)-F_{2;f,q}(v)\big)F_{1;q,j}(v)\\
&-(u-w)\big(F_{2;h,g}(v)-F_{2;h,g}(w)\big)\big(F_{2;f,q}(w)-F_{2;f,q}(v)\big)F_{1;q,j}(u)\\
&+(u-w)F_{2;h,g}(v)\left[F_{2;f,q}(w)\big(F_{1;q,j}(w)-F_{1;q,j}(v)\big)-F_{3,1;f,j}(w)+F_{3,1;f,j}(v)\right]\\
&-(v-w)F_{2;h,g}(v)\left[F_{2;f,q}(w)\big(F_{1;q,j}(w)-F_{1;q,j}(u)\big)-F_{3,1;f,j}(w)+F_{3,1;f,j}(u)\right]\\
&+(v-w)F_{2;h,g}(w)\left[F_{2;f,q}(w)\big(F_{1;q,j}(w)-F_{1;q,j}(u)\big)-F_{3,1;f,j}(w)+F_{3,1;f,j}(u)\right]\\
&-(u-w)F_{2;h,g}(w)\left[F_{2;f,q}(w)\big(F_{1;q,j}(w)-F_{1;q,j}(v)\big)-F_{3,1;f,j}(w)+F_{3,1;f,j}(v)\right]\}
\end{align*}
where $\epsilon=(-1)^{|i|_2|j|_1+|i|_2|h|_3+|j|_1|h|_3+|j|_1|f|_3+|j|_1|g|_2+|g|_2|f|_3}$.

Opening the parentheses of the above identity,
then the part related to $F_{3,1;a,b}$ in the above identity is
\begin{align*}
&uF_{2;h,g}(v)F_{3,1;f,j}(v)+uF_{2;h,g}(w)F_{3,1;f,j}(w)-wF_{2;h,g}(v)F_{3,1;f,j}(v)-vF_{2;h,g}(w)F_{3,1;f,j}(w)\\
&{-}uF_{2;h,g}(v)F_{3,1;f,j}(w){-}uF_{2;h,g}(w)F_{3,1;f,j}(v){-}vF_{2;h,g}(v)F_{3,1;f,j}(u){-}wF_{2;h,g}(w)F_{3,1;f,j}(u)\\
&{+}wF_{2;h,g}(v)F_{3,1;f,j}(u){+}vF_{2;h,g}(w)F_{3,1;f,j}(u){+}vF_{2;h,g}(v)F_{3,1;f,j}(w){+}wF_{2;h,g}(w)F_{3,1;f,j}(v),
\end{align*}
and the other part is
\begin{align*}
	&vF_{2;h,g}(w)F_{2;f,q}(w)F_{1;q,j}(w)+wF_{2;h,g}(v)F_{2;f,q}(v)F_{1;q,j}(v)\\
	&-uF_{2;h,g}(v)F_{2;f,q}(v)F_{1;q,j}(v)	-uF_{2;h,g}(w)F_{2;f,q}(w)F_{1;q,j}(w)\\
	&+uF_{2;h,g}(w)F_{2;f,q}(v)F_{1;q,j}(v)+uF_{2;h,g}(v)F_{2;f,q}(w)F_{1;q,j}(w)\\
	&-wF_{2;h,g}(w)F_{2;f,q}(v)F_{1;q,j}(v)-vF_{2;h,g}(v)F_{2;f,q}(w)F_{1;q,j}(w)\\
	&-uF_{2;h,g}(v)F_{2;f,q}(w)F_{1;q,j}(u)-uF_{2;h,g}(w)F_{2;f,q}(v)F_{1;q,j}(u)\\
	&+uF_{2;h,g}(v)F_{2;f,q}(v)F_{1;q,j}(u)+uF_{2;h,g}(w)F_{2;f,q}(w)F_{1;q,j}(u)\\
	&-wF_{2;h,g}(v)F_{2;f,q}(v)F_{1;q,j}(u)-vF_{2;h,g}(w)F_{2;f,q}(w)F_{1;q,j}(u)\\
	&+wF_{2;h,g}(w)F_{2;f,q}(v)F_{1;q,j}(u)+vF_{2;h,g}(v)F_{2;f,q}(w)F_{1;q,j}(u)
\end{align*}
It follows that the resulting expression is indeed symmetric in $v$ and $w$,
and then the relation \eqref{super-F} holds.

To prove \eqref{superserre-F},
taking its coefficient of $u^{-r}v^{-s}w^{-t}$ in \eqref{super-F}, we have
\begin{align}\label{rst-coeffi}
	\left[\,[F_{a;i,j}^{(r)},F_{b;h,k}^{(s)}],F_{b;f,g}^{(t)}\,\right]+\left[\,[F_{a;i,j}^{(r)},F_{b;f,g}^{(t)}],F_{b;h,k}^{(s)}\,\right]=0,\quad\mbox{if~}|a-b|=1.
\end{align}
Then taking the $u^{-r}v^{-2t}$-coefficient in \eqref{serre-F} in conjunction with \eqref{rst-coeffi} gives the following \eqref{rtt-coeffi-F}.
\begin{equation}\label{rtt-coeffi-F}
	\big[[F_{a;i,j}^{(r)},F_{b;h,k}^{(t)}],F_{b;f,g}^{(t)}\big]=0.\end{equation}

The general case can be achieved by the map $\psi$,
so it suffices to prove the following special case of \eqref{superserre-F} for $n=4$:
\begin{equation}\label{superserre-spe}
	\big[\,[F_{1;f_1,i}^{(r)},F_{2;j,f_2}^{(1)}],[F_{2;g_1,h}^{(1)},F_{3;k,g_2}^{(s)}]\,\big]=0.
\end{equation}
To show \eqref{superserre-spe},
we may assume $f_1=f_2=f$ and $g_1=g_2=g$ by \eqref{F1F2}.
By \eqref{rtt-coeffi-F}, \eqref{F1F2} and super-Jacobi identity, we have:
\begin{equation}\label{even-odd}
	\big[\,[F_{1;f,i}^{(r)},F_{2;j,f}^{(1)}],[F_{2;g,h}^{(1)},F_{3;k,g}^{(s)}]\,\big]
	{=}{-}(-1)^{(|g|_3{+}|f|_2)}\delta_{g,j}\delta_{f,h}\big[\,[F_{1;f,i}^{(r)},F_{2;j,f}^{(1)}],[F_{2;g,h}^{(1)},F_{3;k,g}^{(s)}]\,\big],
\end{equation}
and then further assume that 
$g=j$ and $f=h$.
Thence it is enough to prove the following 
\begin{equation}\label{superserre-sspe}
	\big[\,[F_{1;f,i}^{(r)},F_{2;g,f}^{(1)}],[F_{2;g,f}^{(1)},F_{3;k,g}^{(s)}]\,\big]=0,
\end{equation}
and we consider the following two cases:
(i) $F_{2;g,f}^{(1)}$ is odd;
(ii) $F_{2;g,f}^{(1)}$ is even.
The proof of (i) is the same argument as in the proof of  (\cite[Lemma 6.4]{Peng16}),
and the case (ii) in $\Char\kk=p\neq 2$ is archieved by \eqref{even-odd}.
Now we assume that $p=2$ and then also Peng's argument (\cite[Lemma 6.4]{Peng16}) yields it.
\end{proof}

Then with the identity \eqref{gv-gu/u-v},
we can converts the relations in series form in Lemma \ref{identity-mu} into the desired form in the following Theorem \ref{parabolic},
which generalize \cite[Theorem 7.2]{Peng16} to a positive characteristic.
We need to state that the relations \eqref{coeffi-d}--\eqref{coeffi-superserre-F} are enough as defining relations of $Y_{\mu}(\so)$.
\begin{theorem}\label{parabolic}
Let $\mu=(\mu_1,\ldots,\mu_n)$ be a composition of $M+N$ and $\so$ be a $0^M1^N$-sequence. Associated to this $\mu$ and $\so$, the super Yangian $Y_\mu(\so)$ is generated by the parabolic generators
\begin{align*}
	&\lbrace D_{a;i,j}^{(r)}, D_{a;i,j}^{\prime(r)} \,|\, 1\leq a\leq n, 1\leq i,j\leq \mu_a, r\geq 0\rbrace,\\
	&\lbrace E_{a;i,j}^{(r)} \,|\, 1\leq a< n, 1\leq i\leq \mu_a, 1\leq j\leq\mu_{a+1}, r\geq 1\rbrace,\\
	&\lbrace F_{a;i,j}^{(r)} \,|\, 1\leq a< n, 1\leq i\leq\mu_{a+1}, 1\leq j\leq \mu_a, r\geq 1\rbrace,
\end{align*}
subject only to the relations \eqref{coeffi-d}-\eqref{coeffi-superserre-F}.	
\begin{eqnarray}
	\label{coeffi-d}D_{a;i,j}^{(0)}&=&\delta_{ij}\,,\\
	\label{coeffi-d-1}\sum_{p=1}^{\mu_a}\sum_{t=0}^{r}D_{a;i,p}^{(t)}D_{a;p,j}^{\prime (r-t)}&=&\delta_{r0}\delta_{ij},\end{eqnarray}
\begin{multline}\label{coeffi-d-2}
\big[D_{a;i,j}^{(r)},D_{b;h,k}^{(s)}\big]=
\delta_{ab}(-1)^{|i|_a|j|_a+|i|_a|h|_a+|j|_a|h|_a}\\
\sum_{t=0}^{min(r,s)-1}\big(D_{a;h,j}^{(t)}D_{a;i,k}^{(r+s-1-t)}-D_{a;h,j}^{(r+s-1-t)}D_{a;i,k}^{(t)}\big),
\end{multline}	

{\allowdisplaybreaks
	\begin{multline}\label{p-daeb}
		[D_{a;i,j}^{(r)}, E_{b;h,k}^{(s)}]
		=\delta_{a,b}\delta_{hj}(-1)^{|h|_a|j|_a}\sum_{p=1}^{\mu_a}\sum_{t=0}^{r-1} D_{a;i,p}^{(t)} E_{b;p,k}^{(r+s-1-t)}\\
		-\delta_{a,b+1}(-1)^{|h|_b|k|_a+|h|_b|j|_a+|j|_a|k|_a} \sum_{t=0}^{r-1} D_{a;i,k}^{(t)} E_{b;h,j}^{(r+s-1-t)},
	\end{multline}
	\begin{multline}\label{p-dafb}
		[D_{a;i,j}^{(r)}, F_{b;h,k}^{(s)}]
		=\delta_{a,b}(-1)^{|i|_a|j|_a+|h|_{a+1}|i|_a+|h|_{a+1}|j|_a}\sum_{p=1}^{\mu_a}\sum_{t=0}^{r-1} F_{b;h,p}^{(r+s-1-t)}D_{a;p,j}^{(t)}\\
		+\delta_{a,b+1}(-1)^{|h|_a|k|_b+|h|_a|j|_a+|j|_a|k|_b} \sum_{t=0}^{r-1} F_{b;i,k}^{(r+s-1-t)}D_{a;h,j}^{(t)},
	\end{multline}        
	\begin{multline}\label{p-eafb}
		[E_{a;i,j}^{(r)} , F_{b;h,k}^{(s)}]
		=\delta_{a,b}(-1)^{|h|_{a+1}|k|_a+|j|_{a+1}|k|_a+|h|_{a+1}|j|_{a+1}+1}
		\sum_{t=0}^{r+s-1} D_{a;i,k}^{\prime (r+s-1-t)} D_{a+1;h,j}^{(t)},       
	\end{multline}	
\begin{multline}\label{p-eaea}
	[E_{a;i,j}^{(r)} , E_{a;h,k}^{(s)}]
	=(-1)^{|h|_{a}|j|_{a+1}+|j|_{a+1}|k|_{a+1}+|h|_{a}|k|_{a+1}}\times\\
	\big( \sum_{t=1}^{s-1} E_{a;i,k}^{(r+s-1-t)} E_{a;h,j}^{(t)} 
	-\sum_{t=1}^{r-1} E_{a;i,k}^{(r+s-1-t)} E_{a;h,j}^{(t)}  \big),       
\end{multline} 
\begin{multline}\label{p-fafa}
	[F_{a;i,j}^{(r)} , F_{a;h,k}^{(s)}]
	=(-1)^{|h|_{a+1}|j|_{a}+|j|_{a}|k|_{a}+|h|_{a+1}|k|_{a}}\times\\
	\big( \sum_{t=1}^{r-1} F_{a;i,k}^{(r+s-1-t)} F_{a;h,j}^{(t)} 
	-\sum_{t=1}^{s-1} F_{a;i,k}^{(r+s-1-t)} F_{a;h,j}^{(t)}  \big),         
\end{multline} 
\begin{equation}\label{p-ee}
[E_{a;i,j}^{(r+1)}, E_{a+1;h,k}^{(s)}]-[E_{a;i,j}^{(r)}, E_{a+1;h,k}^{(s+1)}]
	=(-1)^{|j|_{a+1}|h|_{a+1}}\delta_{h,j}\sum_{q=1}^{\mu_{a+1}}E_{a;i,q}^{(r)}E_{a+1;q,k}^{(s)}\,,
\end{equation}
\begin{multline}\label{p-ff}
	[F_{a;i,j}^{(r+1)}, F_{a+1;h,k}^{(s)}]-[F_{a;i,j}^{(r)}, F_{a+1;h,k}^{(s+1)}]=\\
	(-1)^{|i|_{a+1}(|j|_{a}+|h|_{a+2})+|j|_a|h|_{a+2}+1}\delta_{i,k}\sum_{q=1}^{\mu_{a+1}}F_{a+1;h,q}^{(s)}F_{a;q,j}^{(r)}\,,
\end{multline}
\begin{align}
	\label{pc-ee}&[E_{a;i,j}^{(r)}, E_{b;h,k}^{(s)}] = 0
	\qquad\qquad\text{\;\;if\;\; $|b-a|>1$ \;\;or\;\; \;if\;\;$b=a+1$ and $h \neq j$},\\[3mm]
	\label{pc-ff}&[F_{a;i,j}^{(r)}, F_{b;h,k}^{(s)}] = 0
	\qquad\qquad\text{\;\;if\;\; $|b-a|>1$ \;\;or\;\; \;if\;\;$b=a+1$ and $i \neq k$},\\[3mm]\end{align}}
\begin{equation}\label{coeffi-serre-E}
	\big[[E_{a;i,j}^{(r)},E_{b;h,k}^{(t)}],E_{b;f,g}^{(t)}\big]=0,
\end{equation}
\begin{equation}\label{coeffi-serre-F}
	\big[[F_{a;i,j}^{(r)},F_{b;h,k}^{(t)}],F_{b;f,g}^{(t)}\big]=0,\end{equation}
\begin{equation}\label{coeffi-super-E}\big[E_{a;i,j}^{(r)},[E_{a;h,k}^{(s)},E_{b;f,g}^{(\ell)}]\big]+
	\big[E_{a;i,j}^{(s)},[E_{a;h,k}^{(r)},E_{b;f,g}^{(\ell)}]\big]=0, \end{equation}
\begin{equation}\label{coeffi-super-F}\big[F_{a;i,j}^{(r)},[F_{a;h,k}^{(s)},F_{b;f,g}^{(\ell)}]\big]+
	\big[F_{a;i,j}^{(s)},[F_{a;h,k}^{(r)},F_{b;f,g}^{(\ell)}]\big]=0\end{equation}
where $|a-b|=1$.	
\begin{equation}\label{coeffi-superserre-E}
	\big[\,[E_{a;i,f_1}^{(r)},E_{a+1;f_2,j}^{(1)}]\,,\,[E_{a+1;h,g_1}^{(1)},E_{a+2;g_2,k}^{(s)}]\,\big]=0,\end{equation}
\begin{equation}\label{coeffi-superserre-F}\big[\,[F_{a;i,f_1}^{(r)},F_{a+1;f_2,j}^{(1)}]\,,\,[F_{a+1;h,g_1}^{(1)},F_{a+2;g_2,k}^{(s)}]\,\big]=0\end{equation}
where $n\geq 4$.
	If $D_{a;i,j}^{(r)}$ appears on the left-hand side of the equation, then it holds for all $1\leq i,j\leq \mu_{a}$ and all $r\geq 0$; if $E_{a;h,k}^{(s)}$ appears on the left-hand side of the equation, then it holds for all $1\leq h\leq \mu_a, 1\leq k\leq \mu_{a+1}$ and all $s\geq 1$; if $F_{a;f,g}^{(\ell)}$ appears on the left-hand side of the equation, then it holds for all $1\leq g\leq \mu_a, 1\leq f\leq \mu_{a+1}$ and all $\ell\geq 1$.
\end{theorem}
\begin{remark}\label{rk:coeffi}
\begin{enumerate}
\item Different from \cite[Theorem 7.2]{Peng16},
		we need consider all  $F_{a+1;g,f}^{(1)}$ for the {\it quartic Serre relations} \eqref{coeffi-superserre-E} and \eqref{coeffi-superserre-F}.
		If $p\neq 2$, then the proof in Lemma \ref{identity-mu} implies that the quartic Serre relations
		for $|g|_3+|f|_2=0$ already follow from the quadratic and cubic relations,
		however we cannot derive the quartic Serre relations from others in the case $p=2$. 
		
\item 
	Also relations (7.13)-(7.14) of {\it loc. cit.} are expressed here as the four relations \eqref{coeffi-serre-E}-\eqref{coeffi-super-F}.
	This is essential in characteristic $2$.
	
\item  The factor $\delta_{jk}$ seems to be forgotten in the first term in the right side of (7.5) in \cite[Theorem 7.2]{Peng16}. 
We also note that it is enough to consider the case $f_2=h, g_1=j$ in the relations (7.15)-(7.16) in \cite[Theorem 7.2]{Peng16}.
\end{enumerate}
\end{remark}
\begin{proof}
Now we consider the second part of the proof.
Let $\widehat{Y}_{\mu}$ be the abstract superalgebra generated by elements and relations as in the theorem,
where the parities of the generator are given explicitly by \eqref{pad}--\eqref{paf}.
We may further define all the other $E_{a,b;i,j}^{(r)}$ and $F_{a,b;j,i}^{(r)}$ in $\widehat{Y}_{\mu}$ by setting $E_{a,a+1;i,j}^{(r)}:=E_{a;i,j}^{(r)}$ and $F_{a+1,a;i,j}^{(r)}:=F_{a;i,j}^{(r)}$,
then using the formula \eqref{ter} inductively when $|i-j|>1$.
Let $\theta:\widehat{Y}_{\mu}\rightarrow Y_{\mu}$ be the map sending every element in $\widehat{Y}_{\mu}$ into the element in $Y_{\mu}$ with the same name.
The previous paragraph implies that $\theta$ is a well-defined surjective homomorphism.
Therefore, it remains to prove that $\theta$ is also injective.
The injectivity will be proved  similar to the arguments in the complex field
that showing the image of a spanning set for $\widehat{Y}_{\mu}$ under $\theta$ 
is linearly independent in $Y_{\mu}$.

To do that, it works in related graded superalgebras.
Let $\widehat{Y}^{0}_{\mu}$ (respectively, $\widehat{Y}^{+}_{\mu}$, $\widehat{Y}^{-}_{\mu}$)denote the subalgebras of $\widehat{Y}_{\mu}$ generated by the elements
$D_{a;i,j}^{(r)}$ (respectively, $E_{a,b;i,j}^{(r)}$, $F_{b,a;i,j}^{(r)}$). Deﬁne a ﬁltration on  $\widehat{Y}_{\mu}$  (on $\widehat{Y}^{0}_{\mu}$,
$\widehat{Y}^{+}_{\mu}$ and $\widehat{Y}^{-}_{\mu}$
as well) by 
\begin{equation*}
	\text{deg}(D_{a;i,j}^{(r)})=\text{deg}(E_{a,b;i,j}^{(r)})=\text{deg}(F_{b,a;i,j}^{(r)})=r-1,\qquad\text{for all}\;\; r\ge 1,
\end{equation*}
and denote the associated graded superalgebra by $\gr\widehat{Y}_{\mu}$.
Let $\ovl{F}_{a,b;i,j}^{(r)}$ denote the image of $F_{a,b;i,j}^{(r)}$ in the graded superalgebra $\gr_{r-1}\widehat{Y}_{\mu}^-$. 
Then under the following composition
\[
\gr \widehat{Y}_{\mu}^-\otimes \gr \widehat{Y}_{\mu}^0\otimes \gr \widehat{Y}_{\mu}^+\twoheadrightarrow \gr \widehat{Y}_{\mu} \xrightarrow{\theta} \gr Y_{\mu}\cong U(\gl_{M|N}[x]).
\]
and together with relations \eqref{quasid}--\eqref{quasif},
we have that the images of $\ovl{E}_{a,b;i,j}^{(r)}$, $\ovl{D}_{a;i,j}^{(r)}$ and  $\ovl{F}_{b,a;i,j}^{(r)}$ are $(-1)^{|i|_{a}}e_{n_a+i, n_b+j}x^{r-1}$, $(-1)^{|i|_{a}}e_{n_a+i,n_a+j}x^{r-1}$ and $(-1)^{|i|_{b}}e_{n_b+i, n_a+j}x^{r-1}$,
respectively,
where $n_a:=\mu_1+\mu_2+\ldots+\mu_a$.
Then in order to proving the  injectivity of $\theta$,
it suffices to show that the following relation \eqref{inj} holds in $\gr \widehat{Y}_{\mu}^+$ for all $r,s\geq 0$:
\begin{multline}\label{inj}
	[\ovl{E}_{a,b;i,j}^{(r)},\ovl{E}_{c,d;h,k}^{(s)}]=(-1)^{|j|_b|h|_c}\delta_{b,c}\delta_{h,j}\ovl{E}_{a,d;i,k}^{(r+s-1)}\\
	-(-1)^{|i|_a|j|_b+|i|_a|h|_c+|j|_b|h|_c}\delta_{a,d}\delta_{i,k}\ovl{E}_{c,b;h,j}^{(r+s-1)},
\end{multline}
where $1\leq a< b\leq n$, $1\leq c< d\leq n$ and
$1\leq i\leq \mu_a$, $1\leq j\leq \mu_b$, $1\leq h\leq \mu_c$, $1\leq k\leq\mu_d$.

The following equations \eqref{inj-5}--\eqref{inj-8} were proven over $\mathbb{C}$ in \cite[Lemma 7.3]{Peng16}.
Using \eqref{p-eaea}, \eqref{p-ee}, \eqref{pc-ee} and \eqref{ter},
they can be obtained by exactly the same proof.
\begin{equation}\label{inj-5}
	[\ovl{E}_{a,a+1;i,j}^{(r)},\ovl{E}_{b,b+1;h,k}^{(s)}]=0,\;\text{if}\;|a-b|\ne 1,
\end{equation}
\begin{equation}\label{inj-6}
	[\ovl{E}_{a,a+1;i,j}^{(r+1)},\ovl{E}_{b,b+1;h,k}^{(s)}]=
	[\ovl{E}_{a,a+1;i,j}^{(r)},\ovl{E}_{b,b+1;h,k}^{(s+1)}],\; \text{if}\; |a-b|=1,
\end{equation}
\begin{equation}\label{inj-7}
	\big[\ovl{E}_{a,a+1;i,j}^{(r)},[\ovl{E}_{a,a+1;h,k}^{(s)},\ovl{E}_{b,b+1;f,g}^{(t)}]\big]{=}
	-\big[\ovl{E}_{a,a+1;i,j}^{(s)},[\ovl{E}_{a,a+1;h,k}^{(r)},\ovl{E}_{b,b+1;f,g}^{(t)}]\big],\text{if}\, |a{-}b|{=}1,
\end{equation}
\begin{multline}\label{inj-8}
	\ovl{E}_{a,b;i,j}^{(r)}=(-1)^{|h|_{b-1}}[\ovl{E}_{a,b-1;i,h}^{(r)},\ovl{E}_{b-1,b;h,j}^{(1)}]
	=(-1)^{|k|_{a+1}}[\ovl{E}_{a,a+1;i,k}^{(1)},\ovl{E}_{a+1,b;k,j}^{(r)}],
\end{multline}
for all $b>a+1$ and any $1\leq h\leq\mu_{b-1}$, $1\leq k\leq\mu_{a+1}$.
Then the proof of \eqref{inj} is the same with the proof of \cite[Lemma 7.5]{Peng16} except the following relations  \eqref{inj-1}--\eqref{inj-4}(\,cf.\,\cite[Lemma 7.4,(a)--(d)]{Peng16}):
\begin{equation}\label{inj-1}
	[\ovl{E}_{a,a+2;i,j}^{(r)},\ovl{E}_{a+1,a+2;h,k}^{(s)}]=0,\;\; \text{for all}\;\; 1\leq a\leq n-2,
\end{equation}
\begin{equation}\label{inj-2}
	[\ovl{E}_{a,a+1;i,j}^{(r)},\ovl{E}_{a,a+2;h,k}^{(s)}]=0,\;\; \text{for all}\;\; 1\leq a\leq n-2,
\end{equation}
\begin{equation}\label{inj-3}
	[\ovl{E}_{a,a+2;i,j}^{(r)},\ovl{E}_{a+1,a+3;h,k}^{(s)}]=0,\;\; \text{for all}\;\; 1\leq a\leq n-3,
\end{equation}
\begin{equation}\label{inj-4}
	[\ovl{E}_{a,b;i,j}^{(r)},\ovl{E}_{c,c+1;h,k}^{(s)}]=0,\;\; \text{for all}\;\; 1\leq a\leq c<b\leq n.
\end{equation}
The relation \eqref{inj-8} gives the first identity
\begin{eqnarray*}
	(-1)^{|h|_{a+1}}[\ovl{E}_{a,a+2;i,j}^{(r)},\ovl{E}_{a+1,a+2;h,k}^{(s)}]&=&[[\ovl{E}_{a,a+1;i,h}^{(r)},\ovl{E}_{a+1,a+2;h,j}^{(1)}],\ovl{E}_{a+1,a+2;h,k}^{(s)}]\\
	&\overset{(\ref{inj-7})}{=}&-[[\ovl{E}_{a,a+1;i,h}^{(r)},\ovl{E}_{a+1,a+2;h,j}^{(s)}],\ovl{E}_{a+1,a+2;h,k}^{(1)}]\\
	&\overset{(\ref{inj-6})}{=}&-[[\ovl{E}_{a,a+1;i,h}^{(r+s-1)},\ovl{E}_{a+1,a+2;h,j}^{(1)}],\ovl{E}_{a+1,a+2;h,k}^{(1)}],
\end{eqnarray*}
and this last term is zero by \eqref{coeffi-serre-E},
this proves \eqref{inj-1}.
For  \eqref{inj-2}, the same method in \eqref{inj-1} works, except that we apply \eqref{inj-8} on the term $\ovl{E}_{a,a+2;h,k}^{(s)}$.
By applying \eqref{inj-8} on the left side of \eqref{inj-3}, we obtain
\begin{multline*}
[\ovl{E}_{a,a+2;i,j}^{(r)},\ovl{E}_{a+1,a+3;h,k}^{(s)}]=(-1)^{|s|_{a+1}}(-1)^{|t|_{a+1}}\\\big[\,[\ovl{E}_{a,a+1;i,s}^{(r)},\ovl{E}_{a+1,a+2;s,j}^{(1)}]\,,\,
[\ovl{E}_{a+1,a+2;h,t}^{(1)},\ovl{E}_{a+2,a+3;t,k}^{(r)}]\,\big],
\end{multline*}
which is zero by \eqref{coeffi-superserre-E},
hence \eqref{inj-3} is true.

To establish \eqref{inj-4},
we first consider the case $a=c$.
We proceed by induction on $b-a$.
When $b-a=1$ (resp. $b-a=2$), this follows from \eqref{inj-5} (resp. \eqref{inj-2}).
The super-Jacobi identity in conjunction with \eqref{inj-8} gives
\begin{align*}
&(-1)^{|t|_{b-1}}[\ovl{E}_{a,b;i,j}^{(r)},\ovl{E}_{a,a+1;h,k}^{(s)}]\\
=&\big[[\ovl{E}_{a,b-1;i,t}^{(r)},\ovl{E}_{b-1,b;t,j}^{(1)}],\ovl{E}_{a,a+1;h,k}^{(s)}\big]\\
=&\big[\ovl{E}_{a,b-1;i,t}^{(r)},[\ovl{E}_{b-1,b;t,j}^{(1)},\ovl{E}_{a,a+1;h,k}^{(s)}]\big]\\
&{-}({-}1)^{(|i|_{a}{+}|t|_{b-1})(|j|_{b}{+}|t|_{b-1})}\big[\ovl{E}_{b-1,b;t,j}^{(1)},[\ovl{E}_{a,b-1;i,t}^{(r)},\ovl{E}_{a,a+1;h,k}^{(s)}]\big]\\
=&0
\end{align*}
where the last equality follows from the induction hypothesis and \eqref{inj-5}.
For the case $a<c$ in \eqref{inj-4},
we use  \eqref{inj-8} and super-Jacobi identity reduce the problem to showing 
$$
[\ovl{E}_{a,c+1;i,j}^{(r)},\ovl{E}_{c,c+1;h,k}^{(s)}]=0=[\ovl{E}_{a,c+1;i,j}^{(r)},\ovl{E}_{c,c+2;h,k}^{(s)}].
$$
By using again \eqref{inj-8}--\eqref{inj-3},
this follows from the induction on $c-a$.
\end{proof}
\begin{remark}
We consider the case $a=c$ here in the relation \eqref{inj-4},
which is a bit different from (7.24) in \cite[Lemma 7.4]{Peng16},
but it makes no difference to the proof of Theorem \ref{parabolic}.
\end{remark}

We can also obtian the modular version of the ones in \cite[Corollary 7.9]{Peng16}) by the same argument.
\begin{corollary} We have the PBW bases for the following superalgebras.
	\begin{enumerate}
		\item[(1)] The set of supermonomials in $\{ D_{a;i,j}^{(r)}\}_{1\leq a\leq n, 1\leq i,j\leq \mu_a, r\geq 1}$ taken in a certain fixed order forms a basis for $Y_\mu^0$.
		\item[(2)] The set of supermonomials in $\{ E_{a,b;i,j}^{(r)}\}_{1\leq a<b\leq n, 1\leq i\leq\mu_a,1\leq j\leq\mu_b, r\geq 1}$ taken in a certain fixed order forms a basis for $Y_\mu^+$.
		\item[(3)] The set of supermonomials in $\{ F_{b,a;i,j}^{(r)}\}_{1\leq a<b\leq n, 1\leq i\leq \mu_b,1\leq i\leq\mu_a, r\geq 1}$ taken in a certain fixed order forms a basis for $Y_\mu^-$.
		\item[(4)] The set of supermonomials in the union of the elements listed in (1), (2) and (3) taken in a certain fixed order forms a basis for $Y_{\mu}$.
	\end{enumerate}
\end{corollary}

\textbf{Acknowledgment.} 
The author would like to thank Hao Chang for helpful discussions.
This work is supported by the National Natural Science Foundation of China (Grant Nos. 11801394).

\end{document}